\documentclass[a4paper,12pt]{article}

\usepackage {amsfonts, amssymb, amscd,amsthm, amsmath, eucal}

\usepackage{amsbsy}
\usepackage[all]{xy}

\usepackage{color}

\theoremstyle{plain} {
  \newtheorem{thm}{Theorem}[section]
  \newtheorem{defn}[thm]{Definition}
  \newtheorem{cor}[thm]{Corollary}
  \newtheorem{lem}[thm]{Lemma}
  \newtheorem{prop}[thm]{Proposition}
  \theoremstyle{definition}
  \newtheorem{rem}[thm]{Remark}
    \newtheorem{constr}[thm]{Construction}
  \theoremstyle{plain}
  \newtheorem{clm}[thm]{Claim}
  \newtheorem{notation}[thm]{Notation}
  
  \newtheorem{sssection}[thm]{}
}

{
}
\renewcommand{\subsubsection}{\sssection\rm}

\newcommand{\bG}{\mathbf G}

\newcommand{\bu}{\mathbf u}

\renewcommand{\P}{\mathbb P}
\DeclareMathOperator{\spec}{Spec}

\newcommand{\can}{\text{\rm can}}

\newcommand{\id}{\text{\rm id}}
\newcommand{\pr}{\text{\rm pr}}

\newcommand{\inc}{\text{\rm inc}}

\newcommand{\const}{\text{\rm const}}
\newcommand{\Spec}{\text{\rm Spec}}

\newcommand{\Aff}{\mathbf {A}}
\newcommand{\Pro}{\mathbf {P}}

\newcommand \xra {\xrightarrow }

\newcommand \hra {\hookrightarrow }

\newcommand{\ttf}{{\text{f}}}
\renewcommand{\P}{\mathbb P}

\newcommand\mydim{\text{\rm dim}}

\renewcommand \id{\operatorname{id}}

\renewcommand \phi\varphi

\newcommand{\et}{\text{\rm\'et}}

\newcommand{\ZZ}{\mathbb Z}

\begin{document}

\title{Nice triples and a moving lemma for motivic spaces
}

\author{Ivan Panin\footnote{The author acknowledges support of the
RNF-grant 14-11-00456.}
%RFBR grant 13-01-00429-a}
}

%\date{22.11.2012}

%\c{C}
\maketitle

\begin{abstract}
It is proved that for any cohomology theory $A: SmOp/k \to Gr-Ab$ in the sense of \cite{PS}
and any essentially $k$-smooth semi-local $X$
the Cousin complex is exact.
As a consequence we prove that for any integer $n$ the Nisnevich sheaf $\mathcal A^n$,
associated with the presheaf $U\mapsto A^n(U)$, is {\it strictly homotopy invariant}.

Particularly, for any presheaf of $S^1$-spectra $E$ on the category of $k$-smooth schemes
its Nisnevich sheaf of stable $\Aff^1$-homotopy
groups is {\it strictly homotopy invariant}.

The ground field $k$ is arbitrary. We do not use Gabber's presentation lemma. Instead, we use
the machinery of nice triples as invented in \cite{PSV} and developed further in \cite{P3}.
This recovers a known inaccuracy in Morel's arguments in \cite{M}.

The following moving lemma is proved. Let $k$ be a field and $X$ be a quasi-projective variety.
Let $Z$ be a closed subset in $X$ and let $U$ be the semi-local scheme of finitely many closed points
on $X$. Then the natural morphism $U\to X/(X-Z)$ of Nisnevich sheaves is naively $\Aff^1$-homotopic to the constant morphism
of $U\to X/(X-Z)$ sending $U$ to the distinguished point of $X/(X-Z)$.

A refined version of that moving lemma is proved and is used as well.
Moreover, {\it all the above mentioned results are direct consequences of the latter
moving lemma}.
%Particularly, both results
%{\it are true
%if the base field is finite}.
%We do not use Gabber's presentation lemma. Instead, we use
%the machinery of nice triples as invented in \cite{PSV} and developed further in \cite{P3}.
%It turns out that a construction of new nice triples from
%\cite[the proof of lemma 8.1]{OP2} is {\it a crucial one} for our analyses.
%The machinery of nice triples is inspired by the machinery of standard triples due to Voevodsky.

%We apply that moving lemma to prove that for any presheaf of $S^1$-spectra $E$ on the category of $k$-smooth schemes
%its presheaf of stable $\Aff^1$-homotopy
%groups its Cousin complex {\it is exact} locally for the Zarisky topology.
%This recovers a known inaccuracy in Morel's arguments in \cite{Mor}.

%In the series of papers \cite{Pan1}, \cite{Pan2}, \cite{Pan3}
%In a series of papers \cite{Pan0}, \cite{Pan1}, \cite{Pan2}, \cite{Pan3} we give a detailed and better structured
%proof of the Grothendieck--Serre's conjecture
%for semi-local regular rings containing a finite field. The outline of the proof
%is the same as in \cite{P1},\cite{P2},\cite{P3}.
%If the semi-local regular ring contains an infinite field,
%then the conjecture is proved in \cite{FP}. {\it Thus the conjecture
%is true for regular local rings containing a field.
%}
\end{abstract}

\section{Main results}\label{Introduction}
One of the main aim of the paper is to prove the following result.
\begin{thm}\label{StrHomInvIntr}
Let $k$ be a field and let $A: SmOp/k \to Gr-Ab$ be a cohomology theory on
the category $SmOp/k$ in the sense of \cite[Sect.1]{PS}. Let
$\mathcal A^n_{Nis}$ be the Nisnevich sheaf associated with the presheaf
$W\mapsto A^n(W)$. Then $\mathcal A^n_{Nis}$ is homotopy invariant
and even it is strictly homotopy invariant on $(Sm/k)_{Nis}$.
\end{thm}

This theorem is derived in Section \ref{StrHomInvSect} in the standard manner
from the exactness of the Cousin complex associated with
the theory $A$. The exactness of the Cousin complex for the semi-local
essentially $k$-smooth schemes is proved in
Section \ref{cous_complex_section}.
The exactness of the Cousin complex is derived from a moving lemma
(Theorem \ref{An extended_moving_lemma}) stated and proved
in Section \ref{An extended_moving_lemma_sect}.
That moving lemma is easily derived
from Theorem \ref{Major}.
Here is a weaker version of the Theorem \ref{Major},
which shows the shape of Theorem \ref{Major}
(we do not claim that the moving lemma can be derived from
this weaker version).
%and even it can be derived from the following result,
%which is weaker than  Theorem \ref{Major}.
%The proof of Theorem
%\ref{MainHomotopyIntrod}
%is given in Section
%\ref{MHIntrod}.
%For the reader convenience that proof is sketched
%below in the Introduction.

\begin{thm}[Geometric]\label{MajorIntrod}
Let $X$ be an affine $k$-smooth irreducible $k$-variety, and let $x_1,x_2,\dots,x_n$ be closed points in $X$.
Let $U=Spec(\mathcal O_{X,\{x_1,x_2,\dots,x_n\}})$ and $\textrm{f}\in k[X]$ be
a non-zero function vanishing at each point $x_i$. Then
there is a monic polinomial $h\in O_{X,\{x_1,x_2,\dots,x_n\}}[t]$,
a commutative diagram
of schemes with the irreducible affine $U$-smooth $Y$
\begin{equation}
\label{SquareDiagram2_2}
    \xymatrix{
       (\Aff^1 \times U)_{h}  \ar[d]_{inc} && Y_h:=Y_{\tau^*(h)} \ar[ll]_{\tau_{h}}  \ar[d]^{inc} \ar[rr]^{(p_X)|_{Y_h}} && X_f  \ar[d]_{inc}   &\\
     (\Aff^1 \times U)  && Y  \ar[ll]_{\tau} \ar[rr]^{p_X} && X                                     &\\
    }
\end{equation}
and a morphism $\delta: U \to Y$ subjecting to the following conditions:
\begin{itemize}
\item[\rm{(i)}]
the left hand side square
is an elementary {\bf distinguished} square in the category of affine $U$-smooth schemes in the sense of
\cite[Defn.3.1.3]{MV};
\item[\rm{(ii)}]
$p_X\circ \delta=can: U \to X$, where $can$ is the canonical morphism;
\item[\rm{(iii)}]
$\tau\circ \delta=i_0: U\to \Aff^1 \times U$ is the zero section
of the projection $pr_U: \Aff^1 \times U \to U$;
\item[\rm{(iv)}] $h(1)\in \mathcal O[t]$ is a unit.
\end{itemize}
%%%the left hand side square
%%%is an elementary {\bf distinguished} square in the category of affine $U$-smooth schemes in the sense of
%%%\cite[Defn.3.1.3]{MV}. Moreover,
%if
%$pr_U: \Aff^1 \times U \to U$ is the projection,
%then there is a section $\delta: U\to Y$ of the morphism
%$p_U=pr_U\circ \tau$ such that
%$p_X\circ \delta=can: U \to X$, where $can$ is the canonical morphism,
%and $\tau\circ \delta=i_0: U\to \Aff^1 \times U$, where $i_0$ is the zero section
%of the projection $pr_U: \Aff^1 \times U \to U$.
%Finally, $h(1) \in \mathcal O$ is a unit.
\end{thm}
Theorem \ref{Major} is derived from Theorem
\ref{ElementaryNisSquare_1}. Theorem \ref{ElementaryNisSquare_1}
is purely geometric one. Its proof is based on the
theory of nice triples developed in
Sections \ref{ElementaryFibrations}--\ref{Proof_of Theorem equating3_triples}.
We especially stress the significance of Section
\ref{OjPanConstruction}, which contains a construction
of new nice triples and morphisms between them out of certain simple data.
{\it That allows to construct often nice triples with predicted properties}.

The article is organized as follows.
Section \ref{ElementaryFibrations} contains a recollection
about elementary fibrations. In Section \ref{NiceTriples}
we recall definition of nice triples from \cite{PSV}
and inspired by Voevodsky notion of perfect triples.
We state in that section two theorems:
\ref{ElementaryNisSquare_1_triples}
and
\ref{ThEquatingGroups_1}.
In Section \ref{OjPanConstruction} we recall a construction from \cite{OP2}.
It turns out that the construction
plays a crucial role in our analyses.
That construction together with
Proposition \ref {prop:refining_triples}
gives a tool to obtain nice triple with
prescribed properties.
In Section \ref{SectElemNisnevichSquare}
Theorem \ref{ElementaryNisSquare_1_triples} is proved.
In Section \ref{Proof_of Theorem equating3_triples}
Theorem \ref{ThEquatingGroups_1} is proved.
In Section \ref{Reducing Theorem_MainHomotopy} a geometric presentaion
theorem (Theorem \ref{ElementaryNisSquare_1}) is proved
{\it and a stronger version of Theorem \ref{MajorIntrod} is proved too
}
%It is very easy to derive Theorem \ref{Major} in constant group case.
%Nevertherless, that derivation is sketched in the Introduction above.
%In Section \ref{SecEquatingGroups} Theorem \ref{PropEquatingGroups2} on equating group schemes
%is proved.
%In Section \ref{EqGroupSchemeMorph} Theorem \ref{PropEquatingGroups_1_1_1}
%on equating group scheme morphisms is proved. The latter Section is not necessary
%for the present paper. However it is short and very essential for a forthcoming paper
%on Grothendieck--Serre's conjecture \cite{Pan2}.
%In Section \ref{NiceTriplesAndGr_Schemes}  Theorem
%\ref{equating3_triples} is proved.
%Theorem \ref{equating3_triples} states that there is
%a morphism of nice triples equating two morphisms of reductive group schemes.
%In Section \ref{First Application} a first application of the above machinery is given.
%Namely, Theorems \ref{equating3} and \ref{Major2} are proved.
%Theorem \ref{equating3} is one more geometric presentation theorem.
(see Theorem \ref{Major}).

The Appendix A contains the proof of Lemma
\ref{SmallAmountOfPoints}.
The Appendix B contains the proof of the proposition \ref{ArtinsNeighbor}.\\

The author thanks A.Asok, J.Fasel and M.Hoyois for initiation his attention to the topic of
the present paper (especially over finite fields).
The author also thanks A.Stavrova for paying his attention to Poonen's works on Bertini type theorems
for varieties over finite fields. He thanks D.Orlov for useful comments concerning
the weighted projective spaces tacitely involved in the construction of elementary fibrations.
He thanks M.Ojanguren for many inspiring ideas arising from our joint works with him.
The author thanks A.Ananyevskiy, G.Garkusha and A.Neshitov for asking many interesting questions
concerning the topic of the paper.
%Finally, the author thanks the Referee for many valuable comments which resulted
%in the significant improvement of the final version of the present paper.

\section{Elementary fibrations}
\label{ElementaryFibrations}

In this Section we extend a result of M. Artin from~\cite{LNM305}
concerning existence of nice neighborhoods.
The following notion is
%a modification of the one
introduced by Artin in~\cite[Exp. XI, D\'ef. 3.1]{LNM305}.
\begin{defn}
\label{DefnElemFib} An elementary fibration is a morphism of
schemes $p:X\to S$ which can be included in a commutative diagram
\begin{equation}
\label{SquareDiagram}
    \xymatrix{
     X\ar[drr]_{p}\ar[rr]^{j}&&
\overline X\ar[d]_{\overline p}&&Y\ar[ll]_{i}\ar[lld]_{q} &\\
     && S  &\\    }
\end{equation}
of morphisms satisfying the following conditions:
\begin{itemize}
\item[{\rm(i)}] $j$ is an open immersion dense at each fibre of
$\overline p$, and $X=\overline X-Y$; \item[{\rm(ii)}] $\overline
p$ is smooth projective all of whose fibres are geometrically
irreducible of dimension one; \item[{\rm(iii)}] $q$ is finite
\'{e}tale all of whose fibres are non-empty.
\end{itemize}
\end{defn}

\begin{rem}
\label{ElementraryAndAlmostElementary}
Clearly, an elementary fibration is an almost elementary fibration
in the sense of \cite[Defn.2.1]{PSV}.
\end{rem}

We need in the following result, which is a
slight extension of Artin's result \cite[Exp. XI,Prop. 3.3]{LNM305}.
It is proved in the Appendix B.

\begin{prop}
\label{ArtinsNeighbor} Let $k$ be {\bf a finite field}, $X$ be a smooth
{\bf geometrically} irreducible {\bf affine} variety over $k$,
$x_1,x_2,\dots,x_n\in X$ be a family of {\bf closed} points. Then there exists a
Zariski open neighborhood $X^0$ of the family
$\{x_1,x_2,\dots,x_n\}$ and {\bf an elementary fibration} $p:X^0\to
S$, where $S$ is an open sub-scheme of the projective space
$\Pro^{\mydim X-1}$.
\par
If, moreover, $Z$ is a closed co-dimension one subvariety in $X$,
then one can choose $X^0$ and $p$ in such a way that $p|_{Z\bigcap
X^0}:Z\bigcap X^0\to S$ is finite surjective.
\end{prop}

The following result is proved in
\cite[Prop.2.4]{PSV}.
%The proofs of the above Proposition and of the following one are
%provided in Appendix,~\ref{ArtinsNeighb}.

\begin{prop}
\label{CartesianDiagram} Let $p: X \to S$ be an elementary
fibration. If $S$ is a regular semi-local irreducible scheme, then
there exists a commutative diagram of $S$-schemes
\begin{equation}
\label{RactangelDiagram}
    \xymatrix{
X\ar[rr]^{j}\ar[d]_{\pi}&&\overline X\ar[d]^{\overline \pi}&&
Y\ar[ll]_{i}\ar[d]_{}&\\
\Aff^1\times S\ar[rr]^{\text{\rm in}}&&\Pro^1\times S&&
\ar[ll]_{i}\{\infty\}\times S &\\  }
\end{equation}
\noindent such that the left hand side square is Cartesian. Here $j$
and $i$ are the same as in Definition \ref{DefnElemFib}, while
$\pr_S \circ\pi=p$, where $\pr_S$ is the projection $\Aff^1\times
S\to S$.

In particular, $\pi:X\to\Aff^1\times S$ is a finite surjective
morphism of $S$-schemes, where $X$ and $\Aff^1\times S$ are regarded
as $S$-schemes via the morphism $p$ and the projection $\pr_S$,
respectively.
\end{prop}

\section{Nice triples}
\label{NiceTriples}

In the present section we recall and study certain collections
of geometric data and their morphisms. The concept of a {\it nice
triple\/} was introduced in
\cite[Defn. 3.1]{PSV}
and is very similar to that of a {\it standard triple\/}
introduced by Voevodsky \cite[Defn.~4.1]{V}, and was in fact
inspired by the latter notion. Let $k$ be a field, let $X$
be a $k$-smooth
%{\bf GEOMETRICALLY ???}
irreducible {\bf affine} $k$-variety, and let
$x_1,x_2,\dots,x_n\in X$ be {\bf a family of closed points}. Further, let
$\mathcal O=\mathcal O_{X,\{x_1,x_2,\dots,x_n\}}$ be the
corresponding geometric semi-local ring.

After substituting $k$ by its algebraic closure $\tilde k$ in $k[X]$,
we can assume that $X$ is a $\tilde k$-smooth {\it geometrically irreducible} affine $\tilde k$-scheme.
{\bf The geometric irreducibility of $X$ is required in the proposition \ref{ArtinsNeighbor}
to construct an open neighborhood $X^0$ of the family
$\{x_1,x_2,\dots,x_n\}$ and {\bf an elementary fibration} $p:X^0\to S$,
where $S$ is an open sub-scheme of the projective space
$\Pro^{\mydim X-1}_k$. The proposition \ref{ArtinsNeighbor}
is used in turn to prove the proposition \ref{BasicTripleProp_1}.
To simplify the notation, we will
continue to denote this new $\tilde k$ by $k$.
%$U:= \text{Spec}(\mathcal O_{X, \{x_1,x_2,\dots,x_n\}})$.
\begin{defn}
\label{DefnNiceTriple} Let $U:=\text{\Spec}(\mathcal
O_{X,\{x_1,x_2,\dots,x_n\}})$.
%$J= \cap^n_{i=1} \mathfrak m_i, l=\mathcal O_{X, \{x_1,x_2,\dots,x_n\}}/J$.
A \emph{nice triple} over $U$ consists of the following data:
\begin{itemize}
\item[\rm{(i)}] a smooth morphism $q_U:\mathcal X\to U$, where $\mathcal X$ is an irreducible scheme,
\item[\rm{(ii)}] an element $f\in\Gamma(\mathcal X,\mathcal
O_{\mathcal X})$,
%a closed codimension one subscheme
%$\mathcal Z$,
\item[\rm{(iii)}] a section $\Delta$ of the morphism $q_U$,
\end{itemize}
subject to the following conditions:
\begin{itemize}
\item[\rm{(a)}]
%{\bf each component of each fibre of the morphism $q_U$
%is geometrically irreducible of dimension one,}
each irreducible component of each fibre of the morphism $q_U$
has dimension one,
%and can be decomposed as
%$pr_U \circ \Pi$,
%where
%$\Pi: \mathcal X \to \Aff^1 \times U$
%is a finite surjective morphism;
\item[\rm{(b)}]
the module
$\Gamma(\mathcal X,\mathcal O_{\mathcal X})/f\cdot\Gamma(\mathcal X,\mathcal O_{\mathcal X})$
is finite as
a $\Gamma(U,\mathcal O_{U})=\mathcal O$-module,
%$\mathcal Z$ is finite and surjective over $U$;
\item[\rm{(c)}]
there exists a finite surjective $U$-morphism
$\Pi:\mathcal X\to\Aff^1\times U$;
\item[\rm{(d)}]
$\Delta^*(f)\neq 0\in\Gamma(U,\mathcal O_{U})$.
\end{itemize}
There are many choices of the morphism $\Pi$. Any of them is regarded as assigned to
the nice triple.
\end{defn}

\begin{rem}\label{rem: NiceTriples}
Since $\Pi$ is a finite morphism, the scheme $\mathcal X$ is affine.
We will write often below $k[\mathcal X]$ for $\Gamma(\mathcal X,\mathcal O_{\mathcal X})$.
The only requirement on the morphism $\Delta$ is this: $\Delta$ is a section of $q_U$.
Hence $\Delta$ is a closed embedding. We write $\Delta(U)$ for the image of this closed embedding.
The composite map $\Delta^* \circ q^*_U: k[\mathcal X]\to \mathcal O$ is the identity.
If $Ker=Ker(\Delta^*)$, then $Ker$ is the ideal defining the closed subscheme $\Delta(U)$ in $\mathcal X$.
\end{rem}
%{\bf A motivation to require the condition $\Delta^*(f)\neq 0\in\Gamma(U,\mathcal O_{U})$
%rather than the one $\Delta^*(f) \in\Gamma(U,\mathcal O_{U})^{\times}$ is this.
%If for the triple $(q_U:\mathcal X\to U,f,\Delta)$ from Section \ref{MainConstruction} one has
%$\Delta^*(f) \in\Gamma(U,\mathcal O_{U})^{\times}$, then
%$\text{f} \in \mathcal O^{\times}$. So, in this case the principal $G$-bundle
%$P$ is trivial already over $\mathcal O \otimes_k A=\mathcal O_{\text{f}} \otimes_k A$.
%And there is nothing to prove. }

%{\bf A comment is in order here. We use nice triples below to solve the following problem.
%Given a smooth affine $U$-group scheme $H$ and given a principal $H$-bundle $\mathcal H$ over $\mathcal X$,
%which is trivial over $\mathcal X_f$, check that the pull-back $\Delta^*(\mathcal H)$ is trivial
%over $U$ (see Section ?). If $\Delta^*(f) \in\Gamma(U,\mathcal O_{U})^{\times}$, then obviously
%$\Delta^*(\mathcal H)$ is trivial over $U$. So, in this case there is nothing to prove. }
\begin{defn}
\label{DefnMorphismNiceTriple}
A \emph{morphism}
between two nice
triples
over $U$
$$(q^{\prime}: \mathcal X^{\prime} \to U,f^{\prime},\Delta^{\prime})\to(q: \mathcal X \to U,f,\Delta)$$
is an \'{e}tale morphism of $U$-schemes $\theta:\mathcal
X^{\prime}\to\mathcal X$ such that
\begin{itemize}
\item[\rm{(1)}] $q^{\prime}_U=q_U\circ\theta$, \item[\rm{(2)}]
$f^{\prime}=\theta^{*}(f)\cdot h^{\prime}$ for an element
$h^{\prime}\in\Gamma(\mathcal X^{\prime},\mathcal O_{\mathcal
X^{\prime}})$,
%$\theta^{-1} (\mathcal Z) \subset \mathcal Z^{\prime}$
%$\theta^{-1} (\mathcal Z)$ is finite surjective over $U$);
\item[\rm{(3)}] $\Delta=\theta\circ\Delta^{\prime}$.
\end{itemize}
\end{defn}
Two observations are in order here.
\par\smallskip
$\bullet$ Item (2) implies in particular that $\Gamma(\mathcal
X^{\prime},\mathcal O_{\mathcal
X^{\prime}})/\theta^*(f)\cdot\Gamma(\mathcal X^{\prime},\mathcal
O_{\mathcal X^{\prime}})$ is a finite
$\mathcal O$-module.
\par\smallskip
$\bullet$ It should be emphasized that no conditions are imposed
on the interrelation of $\Pi^{\prime}$ and $\Pi$.
\par\smallskip

}
%Note that
%$U\times_{\Spec(\tilde k)}$ as $U$-schemes, and the same holds for $X$ instead of $U$.

%Let $k$ be {\bf a finite field}. Fix a smooth geometrically irreducible affine $k$-scheme $X$,
%and a finite family of {\bf closed} points $x_1,x_2, \dots , x_n$ on $X$, and a
%non-zero function $\ttf\in k[X]$, which vanishes at each of $x_i$'s for $i=1,2,\dots,n$.
%Let $\mathcal O=\mathcal O_{X,\{x_1,x_2,\dots,x_n\}}$ be the semi-local ring of the family $x_1,x_2, \dots , x_n$ on $X$,
%$U=Spec(\mathcal O)$ and
Let $U$ be as in Definition \ref{DefnNiceTriple} and
$\can:U\hra X$ be the canonical inclusion of schemes.
%The definition of a nice triple over $U$ is given in 3.1.
%\ref{DefnNiceTriple}.
%The main aim of the present section is to prove the following
\begin{defn}
\label{SpecialNiceTriples}
A nice triple
$(q_U: \mathcal X \to U, \Delta, f)$ over $U$
is called special if
the set of closed points of $\Delta(U)$ is
contained in the set of closed points of $\{f=0\}$.
\end{defn}

\begin{rem}
\label{NiceToSpecialNice}
Clearly the following holds:
let $(\mathcal X,f,\Delta)$ be a special nice triple
over $U$ and let
$\theta: (\mathcal X^{\prime},f^{\prime},\Delta^{\prime}) \to (\mathcal X,f,\Delta)$
be a morphism between nice triples over $U$. Then the triple
$(\mathcal X^{\prime},f^{\prime},\Delta^{\prime})$
is a special nice triple over $U$.
\end{rem}

\begin{prop}
\label{BasicTripleProp_1}
One can shrink $X$ such that $x_1,x_2, \dots , x_n$ are still in $X$ and $X$ is affine, and then construct a special nice triple
$(q_U: \mathcal X \to U, \Delta, f)$ over $U$ and an essentially smooth morphism $q_X: \mathcal X \to X$ such that
$q_X \circ \Delta= can$, $f=q^*_X(\text{f} \ )$.
%and the set of closed points of $\Delta(U)$ is
%contained in the set of closed points of $\{f=0\}$.
\end{prop}

\begin{proof}[Proof of Proposition \ref{BasicTripleProp_1}]
If the field $k$ is infinite, then this proposition is proved in
\cite[Prop. 6.1]{PSV}.
So, we may and will assume that $k$ is finite.
To prove the proposition repeat literally the proof of
\cite[Prop. 6.1]{PSV}. One has to replace the references to
\cite[Prop. 2.3]{PSV} and
\cite[Prop.2.4]{PSV}
with the reference to
Propositions
\ref{ArtinsNeighbor}
and
\ref{CartesianDiagram}
respectively.

\end{proof}

%\begin{equation}
%\xymatrix{U_K\ar[r]^{\bar\pi}\ar[d]&U\ar[d]\\
%{\spe K}\ar[r]^\pi\ar[ur]^y\ar@/^0.8pc/[u]^{\tilde y}
%\ar@/_0.8pc/[u]_{\tilde x}&{\spe k,}\ar@/_0.8pc/[u]_x}
%\end{equation}

Let us state two crucial results which are used in
Section
\ref{Reducing Theorem_MainHomotopy}
%\ref{Proof_of_Theorem_equating3}
to prove Theorem
\ref{ElementaryNisSquare_1}.
%and
%\ref{equating3}.
Their proofs are given in Sections
\ref{SectElemNisnevichSquare}
and
\ref{Proof_of Theorem equating3_triples}
respectively. If
$U$ as in Definition
\ref{DefnNiceTriple}
then {\it for any $U$-scheme $V$ and any closed point} $u \in U$
set
$$V_u=u\times_U V.$$
For a finite set $A$ denote $\sharp A$ the cardinality of $A$.

\begin{defn}
\label{Conditions_1*and2*}
%Let $V$ be a semi-local $U$-scheme and $p: V\to U$ be its structure morphism and $s: U\to V$ be a section of $p$.
%Assume that for any closed point $v\in V$ the point $p(v)$ is a closed point of $U$.
%Let $s: U\to V$ be its section.
%We say that the pair $(V,s(U))$
Let $(\mathcal X,f,\Delta)$ be a special nice triple
over $U$.
We say that the triple
$(\mathcal X,f,\Delta)$
satisfies conditions $1^*$ and $2^*$
if either the field $k$ is infinite or (if $k$ is finite) the following holds
\begin{itemize}
\item[$(1^*)$]
for $\mathcal Z=\{f=0\} \subset \mathcal X$ and
for any closed point $u \in U$, any integer $d \geq 1$
%and the vanishing locus $\mathcal Z$ of $f$
%closed subscheme $\mathcal Z$ of $\mathcal X$ defined by $\{f=0\}$
one has
$$\sharp\{z \in \mathcal Z_u| \text{deg}[k(z):k(u)]=d \} \leq \ \sharp\{x \in \Aff^1_u| \text{deg}[k(x):k(u)]=d \}$$
\item[$(2^*)$]
for the vanishing locus $\mathcal Z$ of $f$
%$\mathcal X$ defined by $\{f=0\}$
and for any closed point $u \in U$ the point $\Delta(u) \in \mathcal Z_u$ is the only
$k(u)$-rational point of $\mathcal Z_u=u\times_U \mathcal Z$.
\end{itemize}

\end{defn}

\begin{thm}
\label{ElementaryNisSquare_1_triples}
Let $U$ be as in Definition
\ref{DefnNiceTriple}.
Let
$(q^{\prime}_U: \mathcal X^{\prime} \to U,f^{\prime},\Delta^{\prime})$
be a special nice triple
over $U$ subject to the conditions
$(1^*)$ and $(2^*)$ from Definition
\ref{Conditions_1*and2*}.
%such that $f^{\prime}$ {\bf vanishes at every closed point of} $\Delta^{\prime}(U)$.
%Let $\mathcal Z^{\prime}$ be the closed sub-scheme of $\mathcal X^{\prime}$ defined by $\{f^{\prime}=0\}$.
%Assume that $\mathcal Z^{\prime}$ satisfies the conditions {\rm (2)} and {\rm (3)} from
%Theorem
%\ref{ThEquatingGroups_1}.
Then there exists a
finite surjective morphism
$$\Aff^1\times U \xleftarrow{\sigma} \mathcal X^{\prime} $$
\noindent
of $U$-schemes which enjoys the following properties:
\begin{itemize}
\item[\rm{(a)}]
the morphism
$\Aff^1\times U \xleftarrow{\sigma|_{\mathcal Z^{\prime}}} \mathcal Z^{\prime}$
is a closed embedding;
\item[\rm{(b)}] $\sigma$ is \'{e}tale in a neighborhood of
$\mathcal Z^{\prime}\cup\Delta^{\prime}(U)$;
\item[\rm{(c)}]
$\sigma^{-1}(\sigma(\mathcal Z^{\prime}))=\mathcal Z^{\prime}\coprod \mathcal Z^{\prime\prime}$
scheme theoretically and
$\mathcal Z^{\prime\prime} \cap \Delta^{\prime}(U)=\emptyset$;
\item[\rm{(d)}]
$\sigma^{-1}(\{0\} \times U)=\Delta^{\prime}(U)\coprod \mathcal D$ scheme theoretically and $\mathcal D \cap \mathcal Z^{\prime}=\emptyset$;
\item[\rm{(e)}]
for $\mathcal D_1:=\sigma^{-1}(\{1\} \times U)$ one has
%$D^{\prime}_1 \cap \mathcal Z^{\prime}=\emptyset$.
$\mathcal D_1 \cap \mathcal Z^{\prime}=\emptyset$.
\item[\rm{(f)}]
there is a monic polinomial
$h \in \mathcal O[t]$
such that
$(h)=Ker[\mathcal O[t] \xrightarrow{- \circ \sigma^*}
\Gamma(\mathcal X^{\prime}, \mathcal O_{\mathcal X^{\prime}})/(f^{\prime})]$
%{\bf and} $h(1)\in \mathcal O$ {\bf is a unit}.
\end{itemize}
\end{thm}

\begin{thm}
\label{ThEquatingGroups_1}
Let $U$ be as in Definition
\ref{DefnNiceTriple}. Let $(\mathcal X,f,\Delta)$ be a special nice triple
over $U$.
%Let $G_{\mathcal X}$ be a
%semi-simple
%%simply connected
%$\mathcal X$-group scheme, and let $G_U:=\Delta^*(G_{\mathcal
%X})$. Finally, let $G_{\const}$ be the pull-back of $G_U$ to
%$\mathcal X$.
Then there exists a morphism
$\theta:(\mathcal X^{\prime},f^{\prime},\Delta^{\prime})\to(\mathcal X,f,\Delta)$
of nice triples {\bf over} $U$ such that
$(\mathcal X^{\prime},f^{\prime},\Delta^{\prime})$
is a special nice triple satisfying the conditions
$(1^*)$ and $(2^*)$ from Definition
\ref{Conditions_1*and2*}.
\end{thm}

\section{One of the main construction}
\label{OjPanConstruction}
%We finish this section by a construction.
In this section beginning with a nice triple we construct a new one.
Namely,
beginning with a special nice triple
$(\mathcal X,f,\Delta)$
over $U$
we construct a special nice triple
$(\mathcal X',f',\Delta')$
over $U$ and a morphism
$$\theta: (\mathcal X',f',\Delta') \to (\mathcal X,f,\Delta)$$
between nice triples over $U$,
which has, particularly, the following property:\\
if $\Pi:\mathcal X\to\Aff^1\times U$ is a finite surjective morphism assinged to the nice triple,
and $\mathcal Y=\Pi^{-1}(\Pi(\mathcal Z \cup \Delta(U)))$ is the closed subset in
$\mathcal X$, and $y_1,\dots,y_m$ are all its closed points, and
$S=\text{Spec}(\mathcal O_{\mathcal X,y_1,\dots,y_m})$,
and $S'=\theta^{-1}(S)$, {\it then}
$S'$ {\it is \'{e}tale and finite over} $S$,
{\it irreducible and
the set of closed points of the closed subset
}
$\{f'=0\}$
in
{\it $\mathcal X$
is contained in the set of closed points of
}
$S'$.

We begin with recalling the following geometric lemma
\cite[Lemma 8.2]{OP1}.
For reader's convenience we state that Lemma adapting
notation to the ones of the present section.
The following lemma is equivalent to the one
\cite[Lemma 8.2]{OP2}.
\begin{lem}
\label{Lemma_8_2}
Let $U$ be as in Definition \ref{DefnNiceTriple} and let
$(\mathcal X,f,\Delta)$ be a nice triple over $U$.
%Let $\mathcal Y\subset \mathcal X$ be a closed subset respecting to
%a finite surjective morphism
Since $(\mathcal X,f,\Delta)$ is a nice triple over $U$ there is
a finite surjective morphism
$\Pi:\mathcal X\to\Aff^1\times U$ of $U$-schemes.
%(clearly, $\mathcal Y=\Pi^{-1}(\Pi(\mathcal Y))$).
%Set $\mathcal Y:=\Pi^{-1}(\Pi(\mathcal Z \cup \Delta(U)))$.
{\bf Let $\mathcal Y$ be a closed nonempty subset of $\mathcal X$, finite over $U$.
}
Let $\mathcal V$ be an open
subset of $\mathcal X$ containing  $\Pi^{-1}(\Pi(\mathcal Y))$.
Then there
exists an open subset $\mathcal W \subseteq \mathcal V$ still
containing $\Pi^{-1}(\Pi(\mathcal Y))$ and such that
\begin{itemize}
\item
the data $(\mathcal W,f|_{\mathcal W},\Delta)$ is a nice triple over $U$;
\item
the open embedding
$i: \mathcal W \hookrightarrow \mathcal X$
is a morphism
$(\mathcal W,f|_{\mathcal W},\Delta)\to (\mathcal X,f,\Delta)$
between the nice triples over $U$.
\end{itemize}
%%%endowed with a finite
%%%surjective morphism
%%%$\Pi^*: \mathcal W\to\Aff^1\times U$ {\rm(}in general
%%%$\neq\Pi${\rm)}.
\end{lem}
%Choose and fix such a morphism $\Pi$.
%Let $\mathcal Y$ be a closed nonempty sub-scheme
%of $\mathcal X$, finite over $U$. Let $\mathcal V$ be an open
%subset of $\mathcal X$ containing $\Pi^{-1}(\Pi(\mathcal Y))$.
%%%By \cite{OP2}
%%%there
%%%exists an open set $\mathcal W \subseteq \mathcal V$ still
%%%containing $q_U^{-1}(q_U(\mathcal Y))$ and endowed with a finite
%%%surjective morphism
%%%$\Pi^*: \mathcal W\to\Aff^1\times U$
%%%(in general $\Pi^* \neq \Pi$).
Let $\Pi:\mathcal X\to\Aff^1\times U$ be the finite
surjective $U$-morphism as in the lemma
\ref{Lemma_8_2}.
%% from the hypotheses of Lemma \ref{Lemma_8_2}.
%The following diagram summarises the situation:
Consider the following diagram
$$ \xymatrix{
{}&\mathcal Z \ar[d]&{}\\
{\mathcal X - \mathcal Z\ }\ar@{^{(}->}@<-2pt>[r]&\mathcal X\ar@<2pt>[d]^{q_U}\ar[r]^>>>>{\Pi}& \Aff^1 \times U\\
{}& U\ar@<2pt>[u]^{\Delta}&{} } $$
\noindent
Here and in the construction \ref{new_nice_triple}  below
$\mathcal Z$ is the closed
{\bf subset}
defined by the equation
$f=0$. By the assumption, $\mathcal Z$ is finite over $U$.

\begin{constr}(compare  with \cite[the proof of lemma 8.1]{OP2})
\label{new_nice_triple}
Let $U$ be as in Definition \ref{DefnNiceTriple} and let
$(\mathcal X,f,\Delta)$ be a nice triple over $U$.
%Let $\mathcal Z=\{f=0\}$ be the closed subset in $\mathcal X$.
Since $\Delta$ is a section of $q_U$, hence $\Delta(U)$ is a closed subset in
$\mathcal X$.
%%%Let $\mathcal Y\subset \mathcal X$ be a closed subset finite over $U$ respecting to
Let
$\Pi:\mathcal X\to \Aff^1\times U$
a finite surjective morphism
of $U$-schemes, which exists, since
$(\mathcal X,f,\Delta)$ be a nice triple over $U$.
Let
$\mathcal Y=\Pi^{-1}(\Pi(\mathcal Z \cup \Delta(U)))$ be the closed subset in
$\mathcal X$.
Since
$\mathcal Z$ and $\Delta(U)$ are both finite over $U$ and since
$\Pi$ is a finite morphism of $U$-schemes, $\mathcal Y$ is also
finite over $U$. Denote by $y_1,\dots,y_m$ its closed points and
let $S=\text{Spec}(\mathcal O_{\mathcal X,y_1,\dots,y_m})$.
Since $\mathcal Y$ is also
finite over $U$ it is closed in $\mathcal X$.
Since $\mathcal Y$ is contained in $S$
it is also closed in $S$.
%We may assume that $y_1=\Delta(u)$, where $u\in U$ is the closed point.
%Then for any open $V$ in $\mathcal X$ containing $y_1$ one has
%$V\supset \Delta(U)$. So, $\Delta(U)$ is closed in $V$.
%Hence $\Delta(U)$ is in $S$ and it is closed in $S$.
%Set
Set $T=\Delta(U)\subseteq S$.

So, given a nice triple
$(\mathcal X,f,\Delta)$ and the morphism $\Pi$
we get certain data
$(\mathcal Z, \mathcal Y, S, T)$,
where
$\mathcal Z, \mathcal Y, T$
are closed subsets as in $\mathcal X$, so in $S$.
And
$S$ is a semi-local subscheme in $\mathcal X$.
Particularly, the set of all closed points of $\mathcal Z$
{\it is contained
}
in the set of all closed points of $S$.

{\bf Let $\theta_0:S^{\prime}\to S$ be a finite
\'etale morphism with irreducible $S'$ and let
$\delta:T\to S^{\prime}$
be a section
of $\theta_0$ over $T$.
}
We can extend these data to a
neighborhood $\mathcal V$ of $\{y_1,\dots,y_n\}$ in
$\mathcal X$
and get the
diagram
\begin{equation}
\label{Diag_Manuel}
\xymatrix{
     {}  &  S^{\prime} \ar[d]^{\theta_0} \ar@{^{(}->}@<-2pt>[r]  & {\mathcal V}^{\prime}  \ar[d]_{\theta} &\\
     T \ar@{^{(}->}@<-2pt>[r] \ar[ur]^{
\delta } & S \ar@{^{(}->}@<-2pt>[r]  &   \mathcal V \ar@{^{(}->}@<-2pt>[r]  &  \mathcal X &\\
    }
\end{equation}
\noindent
where
$\theta: {\mathcal V}^{\prime}\to\mathcal V$
finite \'etale (and the square is cartesian). Since $T$ isomorphically projects onto $U$, it is still closed
viewed as a sub-scheme of $\mathcal V$. Note that since $\mathcal
Y$ is semi-local and $\mathcal V$ contains all of its closed
points, $\mathcal V$ contains $\Pi^{-1}(\Pi(\mathcal Y))=\mathcal
Y$. By Lemma \ref{Lemma_8_2} there exists an open subset $\mathcal
W\subseteq\mathcal V$ containing $\mathcal Y$ (and hence containing $S$)
and endowed with a
finite surjective $U$-morphism $\Pi^*:\mathcal W\to\Aff^1\times
U$.
Let $\mathcal X^{\prime}=\theta^{-1}(\mathcal W)$,
$f^{\prime}=\theta^{*}(f)$, $q^{\prime}_U=q_U\circ\theta$, and let
$\Delta^{\prime}:U\to\mathcal X^{\prime}$ be the section of
$q^{\prime}_U$ obtained as the composition of $\delta$ with
$\Delta$. This way we get \\
1) firstly, a triple $(\mathcal X^{\prime},f^{\prime},\Delta^{\prime})$;\\
2) secondly, the \'{e}tale morphism of $U$-schemes $\theta: \mathcal X^{\prime} \to \mathcal W \hookrightarrow \mathcal X$;\\
3) thirdly, inclusions of $U$-schemes $S\subset \mathcal W$ and $S' \subset \mathcal X'$ and the equality $S'=\theta^{-1}(S)$.
\end{constr}

\begin{prop}\label{prop:refining_triples}
Under the hypotheses and notation of the construction~\ref{new_nice_triple}
the following is true:
\begin{itemize}
\item[\rm{(i)}]
the triple $(\mathcal X^{\prime},f^{\prime},\Delta^{\prime})$ is a nice triple over $U$;
\item[\rm{(ii)}]
the \'{e}tale morphism
$\theta:\mathcal X^{\prime} \to\mathcal X$
is a morphism between the nice triples
$$\theta: (\mathcal X^{\prime},f^{\prime},\Delta^{\prime}) \to (\mathcal X,f,\Delta) \ \ \text{with} \ \ h^{\prime}=1;$$
\item[\rm{(iii)}]
if the triple
$(\mathcal X,f,\Delta)$
is a special nice triple over $U$, then
the triple
$(\mathcal X^{\prime},f^{\prime},\Delta^{\prime})$
is a special nice triple over $U$;
\item[\rm{(iv)}]
let $\mathcal Z^{\prime}$ be the vanishing locus of $f^{\prime}$ in $\mathcal X^{\prime}$,
then
$\mathcal Z'$ is contained in $S'$ as a closed subset and, particularly,
the set of closed points of $\mathcal Z^{\prime}$ is contained in the set of closed
points of the subscheme $S^{\prime}$.
%the nice triple $(\mathcal X,f,\Delta)$ is special and satisfies the conditions
%$(1^*)$ and $(2^*)$, then the nice triple
%$(\mathcal X^{\prime},f^{\prime},\Delta^{\prime})$
%is a special and satisfies
%the conditions
%$(1^*)$ and $(2^*)$.
\end{itemize}
\end{prop}

\begin{proof}
Firstly, the structure morphism
$q^{\prime}_U:\mathcal X^{\prime} \to U$ coincides with the
composition
$$
\mathcal X^{\prime}\xra{\theta}
\mathcal W\hra\mathcal X\xra{q_U} U.
$$
\noindent
Thus, it is smooth. The element $f^{\prime}$ belongs to the ring
$\Gamma(\mathcal X^{\prime},\mathcal O_{\mathcal X^{\prime}})$,
the morphism $\Delta^{\prime}$ is a section of $q^{\prime}_U$.
Each component of each fibre of the morphism $q_U$ has dimension
one, the morphism $\mathcal X^{\prime}\xra{\theta}\mathcal
W\hra\mathcal X$ is \'{e}tale. Thus, each component of each fibre
of the morphism $q^{\prime}_U$ is also of dimension one. Since
$\{f=0\} \subset {\mathcal W}$ and $\theta: \mathcal X^{\prime}
\to {\mathcal W}$ is finite, $\{f^{\prime}=0\}$ is finite over
$\{f=0\}$ and hence  also over $U$. In other words, the $\mathcal
O$-module $\Gamma(\mathcal X^{\prime},\mathcal O_{\mathcal
X^{\prime}})/f^{\prime}\cdot\Gamma(\mathcal X^{\prime},\mathcal
O_{\mathcal X^{\prime}})$ is finite. The morphism $\theta:
\mathcal X^{\prime}\to\mathcal W$ is finite and surjective.
By the first item of
Lemma \ref{Lemma_8_2}
there is a finite surjective morphism
$\Pi^*:\mathcal W\to\Aff^1\times U$. It follows that
$\Pi^*\circ\theta:\mathcal X^{\prime}\to\Aff^1\times U$ is finite
and surjective.
Hence, the triple
$(\mathcal X^{\prime},f^{\prime},\Delta^{\prime})$ is a nice triple over $U$.
Clearly, the \'{e}tale morphism
$\theta:\mathcal X^{\prime} \to\mathcal X$
is a morphism between the nice triples
$$\theta: (\mathcal X^{\prime},f^{\prime},\Delta^{\prime}) \to (\mathcal X,f,\Delta) \ \ \text{with} \ \ h^{\prime}=1$$
%%%By Remark \ref{NiceToSpecialNice}
%%%the triple
%%%$(\mathcal X^{\prime},f^{\prime},\Delta^{\prime})$
%%%is a special nice triple over $U$ since the one
%%%$(\mathcal X,f,\Delta)$
%%%is a special nice triple over $U$.

If the triple $(\mathcal X,f,\Delta)$
is a special nice triple over $U$, then
by the remark \ref{NiceToSpecialNice}
the triple
$(\mathcal X^{\prime},f^{\prime},\Delta^{\prime})$
is a special nice triple over $U$.
The assertions (i) to (iii) are proved.
%%%This proves the third assertion of the proposition.
%%%\begin{prop}
%%%\label{ref_triple_2}
%%%If the nice triple $(\mathcal X,f,\Delta)$ is special and satisfies the conditions
%%%$(1^*)$ and $(2^*)$, then the nice triple
%%%$(\mathcal X^{\prime},f^{\prime},\Delta^{\prime})$
%%%is a special and satisfies
%%%the conditions
%%%$(1^*)$ and $(2^*)$.
%%%\end{prop}
%%%The nice triple
%%%$(\mathcal X^{\prime},f^{\prime},\Delta^{\prime})$
%%%is a special by the remark
%%%\ref{NiceToSpecialNice}.

To prove the assertion $(iv)$ recall
that
%the set of all closed points of the closed subset
%$\mathcal Z \subset \mathcal X$
%is contained in the set of all closed points of
%$S$. Since
$\mathcal Z$ is
%finite over $U$ it is
a semi-local closed subset as in $\mathcal X$, so in $S$.
The equality $f^{\prime}=\theta^{*}(f)$ shows that $\mathcal Z'=\theta^{-1}(\mathcal Z)$.
Since $S'=\theta^{-1}(S)$, hence $\mathcal Z' \subset S'$.
Since $(\mathcal X^{\prime},f^{\prime},\Delta^{\prime})$ is a nice triple,
hence $\mathcal Z'$ is finite over $U$. Thus
$\mathcal Z'$ is closed in $S'$ and semi-local.
Hence the set of closed points of $\mathcal Z^{\prime}$ is contained in the set of closed
points of the scheme $S^{\prime}$.
%%%By the above choice of
%%%$T^{\prime}=\delta(\Delta(U)) \subset S^{\prime}$
%%%the scheme $T^{\prime}$ projects isomorphically to $U$.
%%%Thus in the case when $k$ is finite the properties
%%%$(1^*)$ and $(2^*)$
%%%of the
%%%$U$-scheme
%%%$S^{\prime\prime}$
%%%show that
%%%the conditions
%%%$(1^*)$ and $(2^*)$
%from Definition
%\ref{Conditions_1*and2*}
%the assertions (2) and (3) of Theorem
%\ref{ThEquatingGroups_1}
%%%are full filled
%%%for the closed sub-scheme
%%%$\mathcal Z^{\prime}$ of $\mathcal X^{\prime}$
%%%defined by $\{f^{\prime}=0\}$.
\end{proof}

\section{Proof of Theorem \ref{ElementaryNisSquare_1_triples} }
\label{SectElemNisnevichSquare}
\begin{proof}[Proof of Theorem \ref{ElementaryNisSquare_1_triples}]
\label{ElementaryNisSquare}
Let $u_1,\ldots,u_n$ be all the closed points of $U$. Let $k(u_i)$ be the residue field of $u_i$. Consider the reduced closed subscheme $\bu$ of $U$,
whose points are $u_1$, \ldots, $u_n$. Thus
\[
 \bu\cong\coprod_i\spec k(u_i).
\]
For any closed point $u \in U$ and any $U$-scheme $V$ let $V_u=u\times_U V$ be
the scheme fibre of the scheme $V$ over the point $u$. Similarly, set
$V_{\bf u}={\bf u}\times_U V$.

Step (i). For any closed point $u \in U$ and any point $z \in \mathcal Z^{\prime}_u$ there is a closed embedding
$z^{(2)} \hra \Aff^1_u$, where
$z^{(2)}:=Spec(\Gamma(\mathcal X^{\prime}_u, \mathcal O_{\mathcal X^{\prime}_u})/\mathfrak m^2_z)$
for the maximal ideal
$\mathfrak m_z \subset \Gamma(\mathcal X^{\prime}_u, \mathcal O_{\mathcal X^{\prime}_u})$
of the point
$z$ regarded as a point of $\mathcal X^{\prime}_u$. This holds, since the $k(u)$-scheme
$\mathcal X^{\prime}_u$
is equi-dimensional of dimension one, affine and $k(u)$-smooth.

Step (ii). For any closed point $u \in U$ there is a closed embedding
$i_u: \coprod_{z \in \mathcal Z^{\prime}_u} z^{(2)} \hra \Aff^1_u$
of the $k(u)$-schemes. To see this apply Step (i) and use that
the triple
$(\mathcal X,f,\Delta)$
satisfies the condition $1^*_U$
%$\mathcal Z^{\prime}$ satisfies the condition {\rm (3)}
from
Definition
\ref{Conditions_1*and2*}.
Set $i_{\bf u}=\sqcup i_u: \sqcup_{u\in {\bf u}} [\coprod_{z \in \mathcal Z^{\prime}_u} z^{(2)}]\hra  \Aff^1_{\bf u}$.

Step(iii) is to introduce some notation.
Since $(\mathcal X^{\prime},f^{\prime},\Delta^{\prime})$ is a nice triple
over $U$ there is a finite surjective morphism
$\mathcal X^{\prime} \xrightarrow{\Pi} \Aff^1\times U$
of the $U$-schemes.
Take the composite
$\mathcal X^{\prime} \xrightarrow{\Pi} \Aff^1\times U \hra \Pro^1\times U$
morphism and denote by
$\bar {\mathcal X^{\prime}}$ the normalization of $\Pro^1\times U$
in the fraction field
$k(\mathcal X^{\prime})$
of the ring
$\Gamma(\mathcal X^{\prime}, \mathcal O_{\mathcal X^{\prime}})$.
The normalization of
$\Aff^1 \times U$ in $k(\mathcal X^{\prime})$
coincides with
the scheme
$\mathcal X^{\prime}$,
since
$\mathcal X^{\prime}$ is a regular scheme.
So, we have a Cartesian diagram of $U$-schemes
\begin{equation}
\label{NormalizationSquare}
    \xymatrix{
     \mathcal X^{\prime}  \ar[rr]^{\inc} \ar[d]_{\Pi} &&
     \bar {\mathcal X^{\prime}} \ar[d]^{\bar \Pi}  &\\
     \Aff^1 \times U \ar[rr]^{\inc} && \Pro^1 \times U, &\\
    }
\end{equation}
in which the horizontal arrows are open embedding.

Let
$\mathcal X^{\prime}_{\infty}$
be the Cartier divisor
$(\bar \Pi)^{-1}(\infty \times U)$
in
$\bar {\mathcal X^{\prime}}$.
Let
$\mathcal L:=\mathcal O_{\bar {\mathcal X^{\prime}}}(\mathcal X^{\prime}_{\infty})$
be the corresponding invertible sheaf and let
$s_0 \in \Gamma(\bar {\mathcal X^{\prime}}, \mathcal L)$
be its section vanishing exactly on
$\mathcal X^{\prime}_{\infty}$.
One has a Cartesian square of $U$-schemes
\begin{equation}
\label{SquareForClosedFibre}
    \xymatrix{
     \mathcal X^{\prime}_{{\infty},{\bf u}}  \ar[rr]^{j_{\infty}} \ar[d]_{in_{\bf u}} &&
     \mathcal X^{\prime}_{\infty} \ar[d]^{in}  &\\
     \bar {\mathcal X^{\prime}_{u,\bf u}} \ar[rr]^{j} && \bar {\mathcal X^{\prime}}, &\\
    }
\end{equation}
which shows that the closed embedding
$in_{\bf u}: \mathcal X^{\prime}_{{\infty},{\bf u}} \hra \bar {\mathcal X^{\prime}_{\bf u}}$
is a Cartier divisor on
$\bar {\mathcal X^{\prime}_{\bf u}}$.
Set
$\mathcal L_{\bf u}=j^*(\mathcal L)$
and
$s_{0,{\bf u}}=s_0|_{\bar {\mathcal X^{\prime}_{\bf u}}} \in \Gamma(\bar {\mathcal X^{\prime}_{\bf u}},\mathcal L_{\bf u})$.

{\bf Step (iii').}
The morphism $\bar \Pi$ is finite surjective. Hence
the morphism
$\bar \Pi|_{\bar {\mathcal X^{\prime}_{\bf u}}}: \bar {\mathcal X^{\prime}_{\bf u}}\to \mathbb P^1_{\bf u}$
is finite surjective.
Hence there is a closed reduced subscheme
$W\subset \bar {\mathcal X^{\prime}_{\bf u}}$
of dimension zero such that
$W\cap \mathcal X^{\prime}_{{\infty}}=\emptyset=W\cap \mathcal Z'$
and $W$ has exactly one point on  each irreducible component of the reduced closed subscheme
$(\mathcal X^{\prime}_{\bf u})_{red}\subset \mathcal X^{\prime}_{\bf u}$.
Let $s_{\infty}: \mathcal O_{\mathcal X^{\prime}_{{\infty}}} \to \mathcal L|_{\mathcal X^{\prime}_{{\infty}}}$
be a trivialization of $\mathcal L|_{\mathcal X^{\prime}_{{\infty}}}$.
Let $t$ be the coordinate function on $\mathbb A^1_{\bf u}$ and
$i_{\bf u}$
%: \coprod_{z \in \mathcal Z^{\prime}_u} z^{(2)} \hra \Aff^1_u$
be the closed embedding from the Step (ii).
Let $J\subset \mathcal O_{\bar {\mathcal X^{\prime}}}$ be a sheaf of
$\mathcal O_{\bar {\mathcal X^{\prime}}}$ ideals
such that
$$\mathcal O_{\bar {\mathcal X^{\prime}}}/J= \mathcal O_{\mathcal X^{\prime}_{{\infty}}} \times
\mathcal O_{W} \times \prod_{u \in {\bf u}}\prod_{z\in \mathcal Z_u}\mathcal O_{\mathcal X^{\prime}_u}/\mathfrak m^2_z.$$ \\
{\bf Claim.}
There exists an integer $n\geq 0$ and a section
$s_1\in H^0(\bar {\mathcal X^{\prime}}, \mathcal L^{\otimes n})$
such that
$$s_1 \text{mod} \ J=(s^n_{\infty},0,i^*_{\bf u}(t)\cdot s^n_0).$$
Prove the Claim. For the coherent sheaf $J$ on $\bar {\mathcal X^{\prime}}$
there is an integer $n(J)\geq 0$ such that for any integer
$n\geq n(J)$ one has
$$H^1(\bar {\mathcal X^{\prime}},J\otimes \mathcal L^{\otimes n})=H^1(\mathbb P^1_U, (\bar \Pi)_*(J)\otimes \mathcal O(n))=0.$$
Thus the map
$H^0(\bar {\mathcal X^{\prime}}, \mathcal L^{\otimes n})=
H^0(\bar {\mathcal X^{\prime}}, \mathcal L^{\otimes n} \otimes \mathcal O_{\bar {\mathcal X^{\prime}}})\to
H^0(\bar {\mathcal X^{\prime}}, \mathcal L^{\otimes n} \otimes (\mathcal O_{\bar {\mathcal X^{\prime}}}/J)$
is surjective. Whence the Claim.

Step (iv). Set $s_{1,{\bf u}}=s_1|_{ \bar {\mathcal X^{\prime}_{\bf u}}}$. Then
$s_{1,{\bf u}} \in \Gamma(\bar {\mathcal X^{\prime}_{\bf u}},\mathcal L^{\otimes n}_{\bf u})$
has no zeros on
$\mathcal X^{\prime}_{{\infty},{\bf u}}$
and the morphism
$$[s^n_{0,{\bf u}}: s_{1,{\bf u}}]: \bar {\mathcal X^{\prime}_{\bf u}} \to \Pro^1_{\bf u}$$
is finite surjective and it is such that: there is an equality $[s^n_{0,{\bf u}}: s_{1,{\bf u}}]^{-1}(\mathbb A^1_{\bf u})=\mathcal X^{\prime}_{\bf u}$ and \\
(a) the morphism $\sigma_{\bf u}= s_{1,{\bf u}}/s^n_{0,{\bf u}}: \mathcal X^{\prime}_{\bf u} \to \Aff^1_{\bf u}$ is finite surjective,\\
(b) $\sigma_{\bf u}|_{\sqcup_{z\in {\bf u}}[\coprod_{z \in \mathcal Z^{\prime}_u} z^{(2)}]}=i_{\bf u}:
\sqcup_{z\in {\bf u}}[\coprod_{z \in \mathcal Z^{\prime}_u} z^{(2)}] \hra \Aff^1_{\bf u}$,
where $i_{\bf u}$ is
from the step (ii); in particular,
$\sigma_{\bf u}$ is \'{e}tale at every point
$z \in \mathcal Z^{\prime}_{\bf u}$.

Step (v).
For any closed point $u \in U$
%There exists a section
%$s_1 \in \Gamma(\bar {\mathcal X^{\prime}}, \mathcal L^{\otimes n})$
one has
$s_1|_{\mathcal X^{\prime}_u}=s_{1,u}$.

Step (vi). If $s_1$ is such as in the step (v), then the morphism
$$\sigma=(s_1/s^n_0, pr_U): \mathcal X^{\prime} \to \Aff^1\times U$$
is finite and surjective. To see this it suffices to check that
the morphism
$([s_0:s_1],pr_U): \bar {\mathcal X^{\prime}}\to \mathbb P^1_U$
is finite surjective and to note that
$[s_0:s_1]^{-1}(\Aff^1)=\mathcal X^{\prime}$.
As we already know the morphism
$[s^n_{0,{\bf u}}: s_{1,{\bf u}}]: \bar {\mathcal X^{\prime}_{\bf u}} \to \Pro^1_{\bf u}$
is finite. Thus the morphism $([s_0:s_1],pr_U)$ is quasi-finite over a neighborhood
of $\Pro^1_{\bf u}$. Since $U$ is semi-local, hence any neighborhood
of $\Pro^1_{\bf u}$ coincides with $\mathbb P^1_U$. Since the morphism
is projective, hence $([s_0:s_1],pr_U)$ is finite and surjective.
Hence, so is the morphism $\sigma$.
%Let
%$q^{\prime}_U: \mathcal X^{\prime} \to U$ be the structural morphism.

{\bf We are ready now to check step by step all the statements of the Theorem.}

{\it The assertion (b)}. Since the schemes $\mathcal X^{\prime}$ and $\Aff^1\times U$
are regular and the morphism $\sigma$ is finite and surjective, the morphism
$\sigma$ is flat by a theorem of Grothendieck
\cite[Thm. 18.17]{E}.

So, to check that $\sigma$ is \'{e}tale at a closed point
$z \in \mathcal Z^{\prime}$
it suffices to check that for the point
$u=q^{\prime}_U(z)$
the morphism
$\sigma_u: \mathcal X^{\prime}_u \to \Aff^1_u$
is \'{e}tale at the point $z$. The latter does hold by
the step (iv), item (b).
Whence $\sigma$ is \'{e}tale at all the closed points of
$\mathcal Z^{\prime}$. By the hypotheses of the Theorem
the set of closed points of
$\Delta^{\prime}(U)$
is contained in the set of the closed points of
$\mathcal Z^{\prime}$.
Whence $\sigma$ is \'{e}tale also at all the closed points of
$\Delta^{\prime}(U)$. The schemes
$\Delta^{\prime}(U)$
and
$\mathcal Z^{\prime}$
are both semi-local.
Thus,
$\sigma$ is \'{e}tale in a neighborhood of
$\mathcal Z^{\prime} \cup \Delta^{\prime}(U)$.

{\it The assertion (a)}.
For any closed point $u \in U$ and
{\bf any point $z \in \mathcal Z^{\prime}_u$ }
the $k(u)$-algebra homomorphism
$\sigma^*_u: k(u)[t] \to k(u)[\mathcal X^{\prime}_u]$
is \'{e}tale at the maximal ideal
$\mathfrak m_z$ of the
$k(u)$-algebra
$k(u)[\mathcal X^{\prime}_u]$
and the composite map
$k(u)[t] \xrightarrow{\sigma^*_u} k(u)[\mathcal X^{\prime}_u] \to k(u)[\mathcal X^{\prime}_u]/\mathfrak m_z$
is an epimorphism. Thus, for any integer $r>0$ the $k(u)$-algebra homomorphism
$k(u)[t] \to k(u)[\mathcal X^{\prime}_u]/\mathfrak m^r_z$
is an epimorphism. The local ring
$\mathcal O_{\mathcal Z^{\prime}_u,z}$
of the scheme
$\mathcal Z^{\prime}_u$
at the point $z$
is of the form
$k(u)[\mathcal X^{\prime}_u]/\mathfrak m^s_z$
for an integer $s$. Thus, the $k(u)$-algebra homomorphism
$$k(u)[t] \xrightarrow{\sigma^*_u} k(u)[\mathcal X^{\prime}_u] \to \mathcal O_{\mathcal Z^{\prime}_u,z}$$
is surjective. Since
$\sigma_u|_{\coprod_{z \in \mathcal Z^{\prime}_u} z^{(2)}}=i_u$
and $i_u$ is a closed embedding one concludes that the
$k(u)$-algebra homomorphism
$$k(u)[t] \to \prod_{z/u}O_{\mathcal Z^{\prime}_u,z}=\Gamma(\mathcal Z^{\prime}_u, \mathcal O_{\mathcal Z^{\prime}_u})$$
is surjective.
Let
${\bf u}=\coprod Spec(k(u))$ regarded as the closed sub-scheme of $U$,
where $u$ runs over all closed points of $U$.
Then, for the scheme
$\mathcal Z^{\prime}_{{\bf u}}={\bf u} \times_U \mathcal Z^{\prime}$
the $k[{\bf u}]$-algebra homomorphism
\begin{equation}
\label{Finiteness}
k[{\bf u}][t] \to \Gamma(\mathcal Z^{\prime}_{{\bf u}}, \mathcal O_{\mathcal Z^{\prime}_{{\bf u}}})
\end{equation}
is surjective.

%Since
%$\Gamma(\mathcal X^{\prime}, \mathcal O_{\mathcal X^{\prime}})$
%is a finitely generated as the
%$\mathcal O[t]$-module.
%Thus
Since $(\mathcal X^{\prime},f^{\prime},\Delta^{\prime})$ is a nice triple over $U$,
the $\mathcal O$-module
$\Gamma(\mathcal Z^{\prime}, \mathcal O_{\mathcal Z^{\prime}})$
is finite.
Thus, the
$k[{\bf u}]$-module
$\Gamma(\mathcal Z^{\prime}_{{\bf u}}, \mathcal O_{\mathcal Z^{\prime}_{{\bf u}}})$
is finite.
Now the surjectivity of the $k[{\bf u}]$-algebra homomorphism
(\ref{Finiteness})
and the Nakayama lemma show that the $\mathcal O$-algebra homomorphism
$\mathcal O[t] \to \Gamma(\mathcal Z^{\prime}, \mathcal O_{\mathcal Z^{\prime}})$
is surjective. Thus,
$\sigma|_{\mathcal Z^{\prime}}: \mathcal Z^{\prime} \to \Aff^1\times U$
is a closed embedding.

{\it The assertion (e)}.
The morphism
$\Delta^{\prime}$
is a section of the structure morphism
$q^{\prime}_U: \mathcal X^{\prime} \to U$
and the morphism
$\sigma$
is a morphism of the $U$-schemes.
Hence
the composite morphism
$t_0:= \sigma \circ \Delta^{\prime}$
is a section of the projection
$pr_U: \Aff^1\times U \to U$.
This section is defined by an element
$a \in \mathcal O$.
There is another section
$t_1$ of the projection
$pr_U$
defined by
the element
$1-a \in \mathcal O$.
Making an affine change of coordinates on
$\Aff^1_U$ we may and will assume that
$t_0(U)=\{0\} \times U$
and
$t_1(U)=\{1\} \times U$.
{
\bf If the field $k$ is infinite we can choose a non-zero
$\lambda\in k$ such that for $t^{new}_1:=\lambda t_1$ one has:
$\mathcal D^{new}_1\cap \mathcal Z^{\prime}=\emptyset$
and
$t^{new}_0:=\lambda t_0=t_0$.
}

Since $\mathcal D_1$ and $\mathcal Z^{\prime}$ are semi-local, to prove the assertion (e)
it suffices to check that
$\mathcal D_1$ and $\mathcal Z^{\prime}$
have no common closed points.
In the infinite field case this is checked just above.
It remains to check the finite field case.
Let $z \in \mathcal D_1 \cap \mathcal Z^{\prime}$ be a common closed point.
Then
$\sigma(z) \in \{1\} \times U$.
Let $u=q^{\prime}_U(z)$. We already know that
$\sigma|_{\mathcal Z^{\prime}}$
is a closed embedding. Thus
$deg[z:u]=deg[\sigma(z):u]=1$.
The triple
$(\mathcal X,f,\Delta)$
satisfies the condition $2^*_U$
from
Definition
\ref{Conditions_1*and2*}.
%The $U$-scheme $\mathcal Z^{\prime}$ satisfies the conditions {\rm (2)}
%of Theorem
%\ref{ThEquatingGroups_1}.
Thus, $z=\Delta^{\prime}(u)$.
In this case
$\sigma(z) \in \{0\} \times U$.
But
$\sigma(z) \in \{1\} \times U$.
This is a contradiction.
Whence
$\mathcal D_1 \cap \mathcal Z^{\prime}=\emptyset$.

{\it The assertion (c)}.
The finite morphism
$\sigma$
is \'{e}tale in a neighborhood of
$\mathcal Z^{\prime}$
by the item (b) of the Theorem.
By the item
(a) of the Theorem
$\sigma|_{\mathcal Z^{\prime}}$
is a closed embedding.
Thus, the morphism
$\sigma^{-1}(\sigma(\mathcal Z^{\prime})) \to \sigma(\mathcal Z^{\prime})$
of affine schemes is finite and
there is an affine open sub-scheme
$V$ of the scheme
$\sigma^{-1}(\sigma(\mathcal Z^{\prime}))$
such that the morphism
$V \to \sigma(\mathcal Z^{\prime})$
is \'{e}tale.
Since
$\sigma|_{\mathcal Z^{\prime}}$
is a closed embedding there is a unique section
$s$ of the morphism
$\sigma^{-1}(\sigma(\mathcal Z^{\prime})) \to \sigma(\mathcal Z^{\prime})$
with the image
$\mathcal Z^{\prime}$
and this image is contained in $V$.
By \cite[Lemma 5.3]{OP1} the scheme
$\sigma^{-1}(\sigma(\mathcal Z^{\prime}))$
has the form
$\sigma^{-1}(\sigma(\mathcal Z^{\prime}))=\mathcal Z^{\prime} \coprod \mathcal Z^{\prime\prime}$.\\
By a similar reasoning the scheme
$\sigma^{-1}(\{0\} \times U)$
has the form
$\Delta^{\prime}(U) \coprod \mathcal D$.
The triple
$(\mathcal X,f,\Delta)$
is a special nice triple. Thus
all the closed points of
$\Delta^{\prime}(U)$
are closed points of
$\mathcal Z^{\prime}$
and
$\mathcal Z^{\prime} \cap \mathcal Z^{\prime\prime}=\emptyset$.
Thus,
$\Delta^{\prime}(U) \cap \mathcal Z^{\prime\prime}=\emptyset$.

{\it The assertion (d)}.
It remains to show that
$\mathcal D \cap \mathcal Z^{\prime}=\emptyset$.
Recall that
$\sigma$ is \'{e}tale in a neighborhood of
$\mathcal Z^{\prime} \cup \Delta^{\prime}(U)$.
{\bf
It's easy to check that
$\sigma|_{\mathcal Z^{\prime} \cup \Delta^{\prime}(U)}$
is a closed embedding. Thus arguing as above one gets
a disjoint union decomposition
$$\sigma^{-1}(\sigma(\mathcal Z^{\prime} \cup \Delta^{\prime}(U)))=(\mathcal Z^{\prime} \cup \Delta^{\prime}(U))\sqcup E$$
for a closed subscheme $E$ in $\mathcal X^{\prime}$.
It's checked already that
$\Delta^{\prime}(U) \coprod \mathcal D=\emptyset$.
Thus, $\mathcal D\subset E$.
Hence
$\mathcal D \cap \mathcal Z^{\prime}=\emptyset$
}.
%It suffices to check that
%$\mathcal D$ and $\mathcal Z^{\prime}$
%have no common closed points.
%Firstly consider the case when the field $k$ is finite.
%Let $z \in \mathcal D \cap \mathcal Z^{\prime}$
%be a common closed point.
%Then
%$\sigma(z) \in \{0\} \times U$.
%Let $u=q^{\prime}_U(z)$. We already know that
%$\sigma|_{\mathcal Z^{\prime}}$
%is a closed embedding. Thus
%$deg[z:u]=deg[\sigma(z):u]=1$.
%The triple
%$(\mathcal X,f,\Delta)$
%satisfies the condition $2^*_U$
%from
%Definition
%\ref{Conditions_1*and2*}.
%The $U$-scheme $\mathcal Z^{\prime}$ satisfies the conditions {\rm (2)}
%of Theorem
%\ref{ThEquatingGroups_1}.
%Thus, $z=\Delta^{\prime}(u) \in \Delta^{\prime}(U)$.
%So,
%$z \in \Delta^{\prime}(U) \cap \mathcal D$.
%But as we already know
%$\Delta^{\prime}(U) \cap \mathcal D=\emptyset$.
%This contradiction shows that
%$\mathcal D$ and $\mathcal Z^{\prime}$
%have no common closed points.
%Thus,
%$\mathcal D \cap \mathcal Z^{\prime}=\emptyset$.

{\it The assertion (f)}.
Recall that
$\mathcal X^{\prime}$ is affine irreducible and regular.
{\bf Thus by Auslander-Goldmann theorem every height one prime ideal is locally principal
and every locally principal ideal
$I\subset \Gamma(\mathcal X^{\prime}, \mathcal O_{\mathcal X^{\prime}})$
is of the form
$I=\mathfrak{q_1}^{a_1}\mathfrak{q_2}^{a_2}...\mathfrak{q_m}^{a_m}$
with integers $a_i\geq 0$.
Moreover, such a presentation of the ideal $I$ is unique.
}
Hence the principal ideal
$(f^{\prime})$ has the form
$\mathfrak p^{r_1}_1\mathfrak p^{r_2}_2\cdots\mathfrak p^{r_n}_n$
with integers $r_i\geq 1$,
where
$\mathfrak p_i$'s
are height one prime ideals in
$k[\mathcal X^{\prime}]:=\Gamma(\mathcal X^{\prime}, \mathcal O_{\mathcal X^{\prime}})$.
Let
$\mathcal Z^{\prime}_i$
be the closed subscheme in
$\mathcal X^{\prime}$
defined by the ideal
$\mathfrak p_i$.
Since $I\subset \mathfrak p_i$, hence
$\mathcal Z^{\prime}_i$ is a closed subscheme in $\mathcal Z'$.

Let
$\mathfrak q_i=\mathcal O[t] \cap \mathfrak p_i$.
The morphism
$\sigma|_{\mathcal Z^{\prime}}: \mathcal Z^{\prime} \to \Aff^1\times U$
is a closed embedding by the item (a) of Theorem
\ref{ElementaryNisSquare_1_triples}.
This yields that
$\sigma|_{\mathcal Z^{\prime}_i}: \mathcal Z^{\prime}_i \to \Aff^1\times U$
is a closed embedding too. Thus
$\mathfrak q_i$
is a hight one prime ideal in
$\mathcal O[t]$.
So, it is a principal prime ideal.
Since
$\mathcal Z^{\prime}$
is finite over
$U$ the scheme
$\mathcal Z^{\prime}_i$
is finite over $U$ too.
Hence the principal prime ideal
$\mathfrak q_i$
is of the form
$(h_i)$ for a unique monic polinomial
$h_i \in \mathcal O[t]$.

Set $h=h^{r_1}_1h^{r_2}_2 \dots h^{r_n}_n$.
Clearly,
$h \in Ker[\mathcal O[t] \xrightarrow{\sigma^*} k[\mathcal X^{\prime}] \xrightarrow{-}
k[\mathcal X^{\prime}]/(f^{\prime})]$.
Since the map
$\mathcal O[t] \xrightarrow{ - \circ \sigma^*} k[\mathcal X^{\prime}]/(f^{\prime})$
is surjective, to prove the assertion (f) it suffices to check that the surjective $\mathcal O$-module homomorphism
$- \circ \sigma^*: \mathcal O[t]/(h) \to k[\mathcal X^{\prime}]/(f^{\prime})$
is an isomorphism.

Consider a decreasing filtration of principal $\mathcal O[t]$-ideals
$$\mathcal O[t]\supset (h_1)\supset ... \supset (h^{r_1}_1)\supset (h^{r_1}_1h_2)\supset ...\supset (h^{r_1}_1h^{r_2}_2)\supset ... \supset (h).$$
Since $\mathcal X^{\prime}$ is an affine scheme there is an element
$g\in k[\mathcal X^{\prime}]$
such that $g|_{\mathcal Z^{\prime\prime}}=0$ and $g|_{\mathcal Z^{\prime}}=1$.
%Set $\mathcal X^{\prime}_0=\mathcal X^{\prime}_g$.
The $\mathcal O[t]$-algebra $k[\mathcal X^{\prime}]$ is flat.
Hence $\mathcal O[t]$-algebra $k[\mathcal X^{\prime}_g]$ is flat.
Thus taking the tensor product of the latter filtration with
the $\mathcal O[t]$-algebra $k[\mathcal X^{\prime}_g]$ we get a decreasing filtration by principal
$k[\mathcal X^{\prime}_g]$-ideals
$$k[\mathcal X^{\prime}_g]\supset h_1\cdot k[\mathcal X^{\prime}_g]\supset ... \supset h\cdot k[\mathcal X^{\prime}_g].$$

The $\mathcal O[t]$-algebra map $\mathcal O[t] \to k[\mathcal X^{\prime}_g]$
induces the map of the filtered $\mathcal O[t]$-modules
$$\mathcal O[t]/(h) \to k[\mathcal X^{\prime}_g]/h\cdot k[\mathcal X^{\prime}_g].$$
It sufficient to check that the induced map on the adjoint graded
$\mathcal O[t]$-modules is an isomorphism. Graded summands for the first
filtration are isomorphic to one of the $\mathcal O[t]$-module
$\mathcal O[t]/(h_i)$.
Graded summands for the second
filtration are isomorphic to one of the $\mathcal O[t]$-module
$k[\mathcal X^{\prime}_g]/(\mathfrak p_i)_g$.
Moreover the map between the corresponding graded summands is the
map of the $\mathcal O[t]$-modules
$\mathcal O[t]/(h_i)\to k[\mathcal X^{\prime}_g]/(\mathfrak p_i)_g$
induced by the scalor extension from $\mathcal O[t]$ to $k[\mathcal X^{\prime}_g]$.
It remains to show that for any $i$ the map of the $\mathcal O$-modules
$\mathcal O[t]/(h_i)\to k[\mathcal X^{\prime}_g]/(\mathfrak p_i)_g$
is an isomorphism.
To show this note that the composite map
$$\mathcal O[t] \xrightarrow{- \circ \sigma^*} k[\mathcal X^{\prime}]/(f^{\prime}) \to
k[\mathcal X^{\prime}]/\mathfrak p_i$$
is an $\mathcal O[t]$-algebra epimorphism (already $\sigma^*$ is an epimorphism). The kernel of the epimorphism
$- \circ \sigma^*$
is the ideal
$\mathfrak q_i=(h_i)$.
Thus $\mathcal O[t]/(h_i)=k[\mathcal X^{\prime}]/\mathfrak p_i$.

Since $g\equiv1 \ \text{mod} (f^{\prime})$ one has equalities
$k[\mathcal X^{\prime}]/(f^{\prime})=k[\mathcal X^{\prime}_g]/f^{\prime}\cdot k[\mathcal X^{\prime}_g]$,
$k[\mathcal X^{\prime}]/\mathfrak p_i=k[\mathcal X^{\prime}_g]/(\mathfrak p_i)_g$.
Hence $\mathcal O[t]/(h_i)=k[\mathcal X^{\prime}]/\mathfrak p_i=k[\mathcal X^{\prime}_g]/(\mathfrak p_i)_g$.

The assertion (f) is proved. Whence the Theorem.
\end{proof}

\section{Proof of Theorems \ref{ThEquatingGroups_1}}
\label{Proof_of Theorem equating3_triples}
Let $k$ be a field.
Let $U$ be as in Definition
\ref{DefnNiceTriple}.
Let $S^{\prime}$ be an irreducible regular semi-local scheme over $k$ and
$p: S^{\prime} \to U$ be a $k$-morphism.
Let
$T^{\prime}\hookrightarrow S^{\prime}$ be a closed sub-scheme of $S^{\prime}$
%$i:T^{\prime}\hookrightarrow S^{\prime}$ be a closed sub-scheme of $S^{\prime}$
such that the restriction
$p|_{T^{\prime}}: T^{\prime} \to U$ is an isomorphism. Let $\delta: U\to T^{\prime}$
be the inverse to $p|_{T^{\prime}}$.
%Let $\delta^{\prime}: U \to S^{\prime}$ be a section of $p$.
%Then $\delta^{\prime}(U) \subset S^{\prime}$ is a closed subscheme of $S^{\prime}$.
%Set $T^{\prime}=\delta^{\prime}(U)$.
We will assume below that $dim(T^{\prime}) < dim(S^{\prime})$,
where $dim$ is the Krull dimension.
For any closed point $u \in U$ and any $U$-scheme $V$ let $V_u=u\times_U V$ be
the fibre of the scheme $V$ over the point $u$.
For a finite set $A$ denote $\sharp A$ the cardinality of $A$.

\begin{lem}
\label{SmallAmountOfPoints}
If the field $k$ is finite and all the closed points of $S^{\prime}$ have {\bf finite residue fields}.
{\bf Suppose that for any closed point $u\in U$ the scheme $S^{\prime}_u$ is a semi-local Dedekind scheme.
}
Then there exists a finite \'{e}tale morphism
$\rho: S^{\prime\prime} \to S^{\prime}$ (with an irreducible scheme $S^{\prime\prime}$)
and a section
$\delta^{\prime}: T^{\prime} \to S^{\prime\prime}$
of $\rho$ over $T^{\prime}$
such that the following holds
\begin{itemize}
\item[\rm{(1)}]
for any closed point $u \in U$ let $u^{\prime} \in T^{\prime}$ be a unique point such that
$p(u^{\prime})=u$, then the point
$\delta^{\prime}(u^{\prime}) \in S^{\prime\prime}_u$ is the only
$k(u)$-rational point of $S^{\prime\prime}_u$,
\item[\rm{(2)}]
for any closed point $u \in U$ and any integer $d \geq 1$
%and for $\mathcal Z^{\prime}$ as in \rm{(2)}
one has
$$\sharp\{z \in S^{\prime\prime}_u| [k(z):k(u)]=d \} \leq \ \sharp\{x \in \Aff^1_u| [k(x):k(u)]=d \}$$
\end{itemize}
If the field $k$ is infinite, then set
$S^{\prime\prime}=S^{\prime}$,
$\rho=id$,
and
$\delta^{\prime}=i$.
%such that the pair
%$(S^{\prime\prime}, \delta^{\prime}(T^{\prime}))$
%subjects to
%the conditions
%$(1^*_U)$ and $(2^*_U)$ from Definition
%\ref{Conditions_1*and2*}.
%\begin{proof}
%If the field $k$ is infinite, then the conditions
%are the empty conditions. In this case take
%$S^{\prime\prime}=S^{\prime}$, $\rho=id$ and $\delta^{\prime}=i$.
%The case of a finite field case we left to the reader.
%\end{proof}
\end{lem}

\begin{proof}
The proof is given at the very end of the Appendix A.
\end{proof}

\begin{proof}[Proof of Theorems \ref{ThEquatingGroups_1}]
Since $(\mathcal X,f,\Delta)$ is a special nice triple
over $U$, there is a finite surjective $U$-morphism $\Pi: \mathcal X \to \mathbb A^1_U$.
Applying to the this nice triple and to the morphism $\Pi$
the first part of the construction \ref{new_nice_triple}
we get the data
$(\mathcal Z,\mathcal Y, S, T)$.

%%%%%The equality
%%%%%(\ref{mu_X_const_0})
%%%%%shows that the morphism $\mu_{\mathcal X}$ is $\theta$-constant (see Definition \ref{};\\
%$\mathcal X^{\prime}$-group scheme isomorphisms
%$$\Phi^{\prime}: \theta^*(G_{\const,\mathcal X})\to\theta^*(G_{\mathcal X}|_{\mathcal X^{\prime}}) \ \ \text{and} \ \
%\Psi^{\prime}: \theta^*(C_{\const,\mathcal X}) \to \theta^*(C_{\mathcal X}|_{\mathcal X^{\prime}})$$
%with
%$(\Delta^{\prime})^*(\Phi^{\prime})= id_{G_U}$,
%$(\Delta^{\prime})^*(\Phi^{\prime})= id_{G_U}$
%and with
%\begin{equation}
%\label{PhiAndPsiRelation}
%\theta^*(\mu_{\mathcal X}) \circ \Phi^{\prime} = \Psi^{\prime} \circ \theta^*(\mu_{const}).
%\end{equation}
%\end{itemize}
%$$ \Phi_0:\theta^*_0(G_{\const,S})\to\theta^*_0(G_{\mathcal X}|_S) $$
%\noindent
%such that $\delta^*\Phi_0=\varphi$.
%%%{\bf Replacing $S^{\prime}$ with a connected component of $S^{\prime}$ which contains
%%%$T^{\prime}:=\delta(T)=\delta(\Delta(U))$
%%%we may and will assume that $S^{\prime}$ is irreducible.}
Let $p=q_U|_{S}: S \to U$ and $\delta=\Delta: U\to S$.
Applying the lemma \ref{SmallAmountOfPoints} to $S'=S$, $T'=T$ and $\delta$
we get
$S^{\prime\prime}$,
$\rho: S^{\prime\prime}\to S$,
and
$\delta': T \to S^{\prime\prime}$
subjecting to the conditions $(1)$ and $(2)$ from
the lemma \ref{SmallAmountOfPoints}.
Recall that
$\rho: S^{\prime\prime} \to S$ is a finite \'{e}tale morphism
(with an irreducible scheme $S^{\prime\prime}$)
and
$\delta^{\prime}$
is a section of
$\rho$
over
$T\subseteq S$.
%$\delta^{\prime}\circ \rho=i: T^{\prime} \hookrightarrow S^{\prime}$.
%Set
%$\delta^{\prime\prime}=\delta^{\prime} \circ \delta: T \to S^{\prime\prime}$
%and
%$\theta^{\prime}_0=\theta_0 \circ \rho: S^{\prime\prime} \to S$.

Applying the second part of the construction \ref{new_nice_triple}
and also the proposition \ref{prop:refining_triples}
to the special nice triple $(\mathcal X,f,\Delta)$, the finite surjective  morphism
$\Pi$
and to the finite
\'etale morphism $\rho$ and to its section
$\delta$
over $T$
we get \\
(i') the nice triple $(\mathcal X^{\prime},f^{\prime},\Delta^{\prime})$ over $U$;\\
(ii') the morphism $\theta: (\mathcal X^{\prime},f^{\prime},\Delta^{\prime}) \to (\mathcal X,f,\Delta)$
between the special nice triples;\\
%a morphism $\theta^{\prime}: \mathcal X^{\prime\prime} \to \mathcal X^{\prime}$
%is a morphism
%between the nice triples $(\mathcal X^{\prime\prime},f^{\prime\prime},\Delta^{\prime\prime})$
%and $(\mathcal X^{\prime},f^{\prime},\Delta^{\prime})$;\\
(iii') the equality $f^{\prime}=(\theta)^*(f)$;\\
(iv') the vanishing locus
$\mathcal Z^{\prime}$ of $f^{\prime}$
on
$\mathcal X^{\prime}$
such that its set of closed points is contained in the set of closed
points of the subscheme $S^{\prime}$.

%%%By the remark \ref{rem:theta_circ_theta'} the morphism
%%%$\theta \circ \theta^{\prime}$ {\it equates} the $\mathcal X$-group scheme morphisms
%%%$\mu_{\const}$
%%%and
%%%$\mu_{\mathcal X}$.
%The properties (iii) and $(iii)'$ and the remark
%\ref{rem:theta_circ_theta'}
%prove that the morphism $\mu_{\mathcal X}$ is
%$(\theta \circ \theta^{\prime})$-constant.
%%%The assertion $(i)$ of the theorem \ref{equating3_triples} is proved.

The properties
$(1)$ and $(2)$
of the
$U$-scheme
$S^{\prime\prime}$
show that
the conditions
$(1^*)$ and $(2^*)$
from the second assertion of
the theorem
\ref{ThEquatingGroups_1}
%\ref{equating3_triples}
%from Definition
%\ref{Conditions_1*and2*}
%the assertions (2) and (3) of Theorem
%\ref{ThEquatingGroups_1}
are full filled
for the closed sub-scheme
$\mathcal Z^{\prime\prime}$ of $\mathcal X^{\prime\prime}$
defined by $\{f^{\prime\prime}=0\}$.
That follows from the property (iv')
mentioned just above.

The property (ii) shows that the triple
$(\mathcal X^{\prime\prime},f^{\prime\prime},\Delta^{\prime\prime}) \to (\mathcal X^{\prime},f^{\prime},\Delta^{\prime})$
is a special nice triple.
This completes the proof of the theorem.
%The second part follows from the remark
%\ref{NiceToSpecialNice}, the lemma
%\ref{SmallAmountOfPoints}
%and the choices of $S^{\prime\prime}$, $\theta^{\prime}_0$ and $\theta^{\prime}$.
%We are also given the
%$\mathcal X^{\prime\prime}$-group scheme isomorphisms
%$$\rho^*(\Phi_0): (\theta^{\prime}_0)^*(G_{\const,S}) \to (\theta^{\prime}_0)^*(G_{\mathcal X}|_S) \ \ \text{and} \ \
%\rho^*(\Psi_0): (\theta^{\prime}_0)^*(C_{\const,S}) \to (\theta^{\prime}_0)^*(C_{\mathcal X}|_S)
%$$
%such that
%$(\delta^{\prime\prime})^*(\rho^*(\Phi_0))=\varphi$,
%$(\delta^{\prime\prime})^*(\rho^*(\Psi_0))=\psi$
%and
%$$(\theta^{\prime})^*_0(\mu_{\mathcal X}|_S) \circ \rho^*(\Phi_0) = \rho^*(\Psi_0) \circ (\theta^{\prime})^*_0(\mu_{\const,S}):
%(\theta^{\prime})^*_0(G_{\const,S}) \to (\theta^{\prime})^*_0(C_{\mathcal X}|_S)$$
\end{proof}

\section{Theorems \ref{ElementaryNisSquare_1} and proof of Theorem \ref{MajorIntrod}}
\label{Reducing Theorem_MainHomotopy}
Theorem \ref{ElementaryNisSquare_1} is a purely geometric one. If the base field $k$
is finite, say, of two elements, if the closed point of $U$ is $k$-rational and the scheme $\mathcal Z'$ below
is such that its closed fibre $\mathcal Z'_u$ contains three $k$-rational points,
then there are no closed embedding $\mathcal Z'$ into $\Aff^1 \times U$.
So, one of the main problem in the proof of the next theorem is to find
such an $\mathcal X$, a morphism $q_X$ and a function $f^{\prime}$
to overcome the mentioned difficulties.
\begin{thm}
\label{ElementaryNisSquare_1}
Let $X$ be an affine $k$-smooth irreducible $k$-variety, and let $x_1,x_2,\dots,x_n$ be closed points in $X$.
Let $U=Spec(\mathcal O_{X,\{x_1,x_2,\dots,x_n\}})$.
Given a non-zero function $\textrm{f}\in k[X]$ vanishing at each point $x_i$,
there is a diagram of the form
\begin{equation}
\label{DeformationDiagram0}
    \xymatrix{
\Aff^1 \times U\ar[drr]_{\pr_U}&&\mathcal X \ar[d]^{}
\ar[ll]_{\sigma}\ar[d]_{q_U}
\ar[rr]^{q_X}&&X &\\
&&U \ar[urr]_{\can}\ar@/_0.8pc/[u]_{\Delta} &\\
    }
\end{equation}
with an irreducible {\bf affine} scheme $\mathcal X$, a smooth morphism $q_U$, a finite surjective $U$-morphism $\sigma$ and an essentially smooth morphism $q_X$,
and a function $f^{\prime} \in q^*_X(\textrm{f} \ )k[\mathcal X]$,
%$f^{\prime}=\theta^{*}(f)\cdot h^{\prime}$ for an element
%$h^{\prime}\in\Gamma(\mathcal X^{\prime},\mathcal O_{\mathcal X^{\prime}})$,
which enjoys the following properties:
\begin{itemize}
\item[\rm{(a)}]
if
%$f=q^*_X(\textrm{f})$ and
$\mathcal Z^{\prime}$ is the closed subscheme of $\mathcal X$ defined by the principal ideal
$(f^{\prime})$, the morphism
$\sigma|_{\mathcal Z^{\prime}}: \mathcal Z^{\prime} \to \Aff^1\times U$
is a closed embedding and the morphism
$q_U|_{\mathcal Z^{\prime}}: \mathcal Z^{\prime} \to U$ is finite;
\item[\rm{(a')}] $q_U\circ \Delta=id_U$ and $q_X\circ \Delta=can$ and $\sigma\circ \Delta=i_0$ \\
(the first equality shows that $\Delta(U)$ is a closed subscheme in $\mathcal X$);
\item[\rm{(b)}] $\sigma$
is \'{e}tale in a neighborhood of
$\mathcal Z^{\prime}\cup \Delta(U)$;
\item[\rm{(c)}]
$\sigma^{-1}(\sigma(\mathcal Z^{\prime}))=\mathcal Z^{\prime}\coprod \mathcal Z^{\prime\prime}$
scheme theoretically
for some closed subscheme $\mathcal Z^{\prime\prime}$
\\ and
$\mathcal Z^{\prime\prime} \cap \Delta(U)=\emptyset$;
\item[\rm{(d)}]
$\mathcal D_0:=\sigma^{-1}(\{0\} \times U)=\Delta(U)\coprod \mathcal D^{\prime}_0$
scheme theoretically
for some closed subscheme $\mathcal D^{\prime}_0$
and $\mathcal D^{\prime}_0 \cap \mathcal Z^{\prime}=\emptyset$;
\item[\rm{(e)}]
for $\mathcal D_1:=\sigma^{-1}(\{1\} \times U)$ one has
%$D^{\prime}_1 \cap \mathcal Z^{\prime}=\emptyset$.
$\mathcal D_1 \cap \mathcal Z^{\prime}=\emptyset$.
\item[\rm{(f)}]
there is a monic polinomial
$h \in \mathcal O[t]$
such that
%$(h)=Ker[\mathcal O[t] \xrightarrow{{\bar \sigma}^*}\Gamma(\mathcal X, \mathcal O_{\mathcal X})/(f^{\prime})]$, \\
$(h)=Ker[\mathcal O[t] \xrightarrow{\sigma^*} k[\mathcal X] \xrightarrow{-} k[\mathcal X]/(f^{\prime})]$, \\
where $\mathcal O:=k[U]$ and the map bar takes any $g\in k[\mathcal X]$ to ${\bar g}\in k[\mathcal X]/(f^{\prime})$.
%is the composite map
%$\mathcal O[t] \xrightarrow{\sigma^*} k[\mathcal X] \to k[\mathcal X]/(f^{\prime})$.
%Here $\mathcal O:=k[U]$.
%$(h)=Ker[\mathcal O[t] \xrightarrow{\sigma^*} \Gamma(\mathcal X, \mathcal O_{\mathcal X}) \to
%\Gamma(\mathcal X, \mathcal O_{\mathcal X})/(f^{\prime})]$.
%and ?? the principal ideal $(h)$ defines the closed subscheme $\sigma(\mathcal Z)$.
\end{itemize}
\end{thm}
%Note that the morphism $\Delta$ is a section of the morphism $q_U$. Thus $\Delta(U)$
%is a closed subscheme of the scheme $\mathcal X$.

%\section{Proof of Theorems \ref{ElementaryNisSquare_1}}
%\label{Proof_of_Theorem_equating3}
\begin{proof}[Proof of Theorem \ref{ElementaryNisSquare_1}]
By Proposition \ref{BasicTripleProp_1} one can shrink $X$ such that
$x_1,x_2, \dots , x_n$ are still in $X$ and $X$ is affine, and then to construct a special nice triple
$(q_U: \mathcal X \to U, \Delta, f)$ over $U$ and an essentially smooth morphism $q_X: \mathcal X \to X$ such that
$q_X \circ \Delta= can$, $f=q^*_X(\text{f})$ and the set of closed points of $\Delta(U)$ is
contained in the set of closed points of $\{f=0\}$.

By Theorem
\ref{ThEquatingGroups_1}
there exists a morphism
$\theta:(\mathcal X^{\prime},f^{\prime},\Delta^{\prime})\to(\mathcal X,f,\Delta)$
such that the triple
$(\mathcal X^{\prime},f^{\prime},\Delta^{\prime})$
is a special nice triple
over $U$
subject to the conditions
$(1^*)$ and $(2^*)$ from Definition
\ref{Conditions_1*and2*}.

The triple
$(\mathcal X^{\prime},f^{\prime},\Delta^{\prime})$
is a special nice triple
{\bf over} $U$
subject to the conditions
$(1^*)$ and $(2^*)$ from Definition
\ref{Conditions_1*and2*}.
%the condition $(2)$ to $(4)$
%from Theorem \ref{ThEquatingGroups_1}
%are full filled.
Thus by Theorem
\ref{ElementaryNisSquare_1_triples}
there is a finite surjective morphism
$\Aff^1\times U \xleftarrow{\sigma} \mathcal X^{\prime}$
of the $U$-schemes satisfying the conditions
$(a)$ to $(\textrm{f})$
from that Theorem. Hence one has a diagram of the form
\begin{equation}
\label{DeformationDiagram0_1}
    \xymatrix{
\Aff^1 \times U\ar[drr]_{\pr_U}&&\mathcal X^{\prime} \ar[d]^{}
\ar[ll]_{\sigma}\ar[d]_{q_U\circ \theta}
\ar[rr]^{q_X\circ \theta}&&X &\\
&&U \ar[urr]_{\can}\ar@/_0.8pc/[u]_{\Delta^{\prime}} &\\
    }
\end{equation}
with the irreducible scheme $\mathcal X^{\prime}$, the smooth morphism $q_U\circ \theta$,
the finite surjective morphism $\sigma$ and the essentially smooth morphism $q_X\circ \theta$
and with the function
$f^{\prime} \in (q_X\circ \theta)^*(\textrm{f})k[\mathcal X^{\prime}]$,
which after identifying notation enjoy the properties
(a) to (\textrm{f}) from Theorem \ref{ElementaryNisSquare_1}.
%Moreover the isomorphism $\Phi$ of reductive $\mathcal X^{\prime}$-group schemes
%is such that $(\Delta^{\prime})^*(\Phi)=\id_{G_U}$.
%The isomorphisms $\Phi$ and $\Psi$ are the required ones.
Whence
the Theorem \ref{ElementaryNisSquare_1}.
\end{proof}

%The proof of this two theorem will be given in Section \ref{Proof_of_Theorem_equating3}.
To formulate a first consequence of the theorem \ref{ElementaryNisSquare_1} (see Corollary \ref{ElementaryNisSquareNew_1}).
note that using the items (b) and (c) of Theorem
\ref{ElementaryNisSquare_1}
one can find an element
$g \in I(\mathcal Z^{\prime\prime})$
such that \\
(1) $(f^{\prime})+(g)=\Gamma(\mathcal X, \mathcal O_{\mathcal X})$, \\
(2) $Ker((\Delta)^*)+(g)=\Gamma(\mathcal X, \mathcal O_{\mathcal X})$, \\
(3) $\sigma_g=\sigma|_{\mathcal X_g}: \mathcal X_g \to \Aff^1_U$ is \'{e}tale.\\

%\begin{lem}
%\label{KeyUnramifiedness_1}
%Let $B \subset A$ be a finite extension of $K$-smooth algebras which are domains
%and each has dimension one,
%Let $0 \neq f \in A$
%%be such that $A/fA$ is finite over $K$.
%% and reduced.
%and let $h \in B\cap fA$ be such that the induced map
%$B/hB\to A/fA$ is an isomorphism.
%Suppose
%%$A/hA=A/fA \times A/J^{\prime\prime}$
%$hA=fA\cap J^{\prime\prime}$
%for an ideal
%$J^{\prime\prime} \subseteq A$
%co-prime to the ideal $fA$.

%%Suppose $N_{B/A}(f)=fg \in B$ for a certain $g \in B$ coprime
%%with $f$. Suppose the composite map
%%$A/N(f)A \to B/N(f)B \to B/fB$
%%is an isomorphism.
%Let $E$ and $F$ be the field of fractions of $B$ and $A$ respectively.
%Let $\alpha \in C(A_f)$ be such that
%$\bar \alpha \in {\cal F}(F)$
%is $A$-unramified. Then, for
%$\beta= N_{F/E}(\alpha)$,
%the class
%$\bar \beta \in {\cal F}(E)$
%is $B$-unramified.
%\end{lem}

\begin{cor}[Corollary of Theorem \ref{ElementaryNisSquare_1}]
\label{ElementaryNisSquareNew_1}
The function $f^{\prime}$ from Theorem \ref{ElementaryNisSquare_1}, the polinomial $h$ from the item $(\textrm{f} \ )$
of that Theorem, the morphism $\sigma: \mathcal X \to \Aff^1_U$
and the function
$g \in \Gamma(\mathcal X,\mathcal O_{\mathcal X} )$
defined just above
enjoy the following properties:
\begin{itemize}
\item[\rm{(i)}]
the morphism
$\sigma_g= \sigma|_{\mathcal X_g}: \mathcal X_g \to \Aff^1\times U $
is \'{e}tale,
\item[\rm{(ii)}]
data
$ (\mathcal O[t],\sigma^*_g: \mathcal O[t] \to \Gamma(\mathcal X,\mathcal O_{\mathcal X})_g, h ) $
satisfies the hypotheses of
\cite[Prop.2.6]{C-TO},
i.e.
$\Gamma(\mathcal X,\mathcal O_{\mathcal X} )_g$
is a finitely generated
%as the
$\mathcal O[t]$-algebra, the element $(\sigma_g)^*(h)$
is not a zero-divisor in
$\Gamma(\mathcal X,\mathcal O_{\mathcal X} )_g$
and
$\mathcal O[t]/(h)=\Gamma(\mathcal X,\mathcal O_{\mathcal X})_g/h\Gamma(\mathcal X,\mathcal O_{\mathcal X})_g$ \ ,
\item[\rm{(iii)}]
$(\Delta(U) \cup \mathcal Z') \subset \mathcal X_g$ \ and $\sigma_g \circ \Delta=i_0: U\to \Aff^1\times U$,
\item[\rm{(iv)}]
$\mathcal X_{gh} \subseteq \mathcal X_{gf^{\prime}}\subseteq \mathcal X_{f^{\prime}}\subseteq \mathcal X_{q^*_X(\textrm{f})}$ \ ,
\item[\rm{(v)}]
$\mathcal O[t]/(h)=\Gamma(\mathcal X,\mathcal O_{\mathcal X})/(f^{\prime})$
and
$h\Gamma(\mathcal X,\mathcal O_{\mathcal X})=(f^{\prime})\cap I(\mathcal Z^{\prime\prime})$
and
$(f^{\prime}) +I(\mathcal Z^{\prime\prime})=\Gamma(\mathcal X,\mathcal O_{\mathcal X})$.
%$\Gamma(\mathcal X,\mathcal O_{\mathcal X})/(h)= \Gamma(\mathcal X,\mathcal O_{\mathcal X})/(f^{\prime}) \
%\times \Gamma(\mathcal X,\mathcal O_{\mathcal X})/I(\mathcal Z^{\prime\prime})$.
%\item[\rm{(e)}]
%$\sigma^{-1}(\{1\}\times U) \cap \mathcal Z^{\prime}=\emptyset$.
\end{itemize}
\end{cor}

\begin{proof}[Proof of Corollary \ref{ElementaryNisSquareNew_1}]
%We need to find $\sigma$, $h$ and $g$
%which enjoy the properties
%(a) to (d) from the Theorem.
%For that
We will use notation from Theorem
\ref{ElementaryNisSquare_1}.
%{\it Take $\sigma$ as in Theorem}
%\ref{ElementaryNisSquare_1}.
Since
$\mathcal X$
is a regular affine irreducible
scheme
and
$\sigma: \mathcal X \to \Aff^1_U$
is finite surjective
the induced $\mathcal O$-algebra homomorphism
$\sigma^*: \mathcal O[t] \to \Gamma(\mathcal X, \mathcal O_{\mathcal X})$
is a monomorphism. We will regard
below the $\mathcal O$-algebra
$\mathcal O[t]$
as a subalgebra via $\sigma^*$.

%{\it Take $h \in \mathcal O[t]$ as in the item
%\emph{(f)} of Theorem}
%\ref{ElementaryNisSquare_1}.

%Let
%$I(\mathcal Z^{\prime\prime})\subseteq \Gamma(\mathcal X^{\prime}, \mathcal O_{\mathcal X^{\prime}})$
%be the ideal defining the closed
%subscheme
%$\mathcal Z^{\prime\prime}$ of the scheme
%$\mathcal X^{\prime}$.
%Using the items (b) and (c) of Theorem
%\ref{ElementaryNisSquare_1}
%find an element
%$g \in I(\mathcal Z^{\prime\prime})$
%such that \\
%(1) $(f^{\prime})+(g)=\Gamma(\mathcal X^{\prime}, \mathcal O_{\mathcal X^{\prime}})$, \\
%(2) $Ker((\Delta^{\prime})^*)+(g)=\Gamma(\mathcal X^{\prime}, \mathcal O_{\mathcal X^{\prime}})$, \\
%(3) $\sigma_g=\sigma|_{\mathcal X^{\prime}_g}: \mathcal X^{\prime}_g \to \Aff^1_U$ is \'{e}tale.\\
%With this choice of $\sigma$, $h$ and $g$ complete the proof of Theorem
%\ref{ElementaryNisSquareNew}.
The assertions (i) and (iii) of the Corollary hold by our choice of $g$.
The assertion (iv) holds, since
$\sigma^*(h)$ is in the principal ideal $(f^{\prime})$
(use the properties (a) and ( \textrm{f}) from Theorem \ref{ElementaryNisSquare_1}).
It remains to prove the assertion (ii).
The morphism $\sigma$ is finite. Hence the
$\mathcal O[t]$-algebra
$\Gamma(\mathcal X, \mathcal O_{\mathcal X})_g$
is finitely generated. The scheme
$\mathcal X$
is regular and irreducible. Thus,
the ring
$\Gamma(\mathcal X, \mathcal O_{\mathcal X})$
is a domain. The homomorphism
$\sigma^*$
is injective. Hence, the element
$h$ is not zero and is not a zero divisor in
$\Gamma(\mathcal X,\mathcal O_{\mathcal X} )_g$.

It remains to check that
$\mathcal O[t]/(h)=\Gamma(\mathcal X,\mathcal O_{\mathcal X} )_g/h\Gamma(\mathcal X,\mathcal O_{\mathcal X} )_g$.
Firstly, by the choice of $h$ and by the item (a) of Theorem
\ref{ElementaryNisSquare_1}
one has
$\mathcal O[t]/(h)=\Gamma(\mathcal X,\mathcal O_{\mathcal X})/(f^{\prime})$.
Secondly, by the property (1) of the element $g$ one has
$\Gamma(\mathcal X,\mathcal O_{\mathcal X})/(f^{\prime})=
\Gamma(\mathcal X,\mathcal O_{\mathcal X})_g/f^{\prime}\Gamma(\mathcal X,\mathcal O_{\mathcal X})_g$.
Finally, by the items (c) and (a) of Theorem
\ref{ElementaryNisSquare_1}
one has
\begin{equation}
\label{Prelocalization_1}
\Gamma(\mathcal X,\mathcal O_{\mathcal X})/(f^{\prime}) \
\times \Gamma(\mathcal X,\mathcal O_{\mathcal X})/I(\mathcal Z^{\prime\prime})=
\Gamma(\mathcal X,\mathcal O_{\mathcal X})/(h).
\end{equation}
Localizing both sides of (\ref{Prelocalization_1}) in $g$ one gets an equality
$$\Gamma(\mathcal X,\mathcal O_{\mathcal X})_g/f^{\prime}\Gamma(\mathcal X,\mathcal O_{\mathcal X})_g=
\Gamma(\mathcal X,\mathcal O_{\mathcal X})_g/h\Gamma(\mathcal X,\mathcal O_{\mathcal X})_g,$$
hence
$\mathcal O[t]/(h)=
\Gamma(\mathcal X,\mathcal O_{\mathcal X})/(f^{\prime})=
\Gamma(\mathcal X,\mathcal O_{\mathcal X})_g/f^{\prime}\Gamma(\mathcal X,\mathcal O_{\mathcal X})_g=
\Gamma(\mathcal X,\mathcal O_{\mathcal X})_g/h\Gamma(\mathcal X,\mathcal O_{\mathcal X})_g.
$
Whence the Corollary.
\end{proof}

%Before proving this Corollary
%Let us make one remark
\begin{rem}
\label{ElementaryNisSquareRem_1}
The item \rm{(ii)} of this corollary shows that the cartesian square
\begin{equation}
\label{SquareDiagram2_1}
    \xymatrix{
     \mathcal X_{gh}  \ar[rr]^{\inc} \ar[d]_{\sigma_{gh}} &&  \mathcal X_g \ar[d]^{\sigma_g}  &\\
     (\Aff^1 \times U)_{h} \ar[rr]^{\inc} && \Aff^1 \times U &\\
    }
\end{equation}
can be used to glue principal $\bG$-bundles for a reductive $U$-group scheme $\bG$.
\end{rem}
Set $Y:=\mathcal X_g$, $p_X=q_X: Y\to X$, $p_U=q_U: Y\to U$, $\tau=\sigma_g$, $\tau_h=\sigma_{gh}$, $\delta=\Delta$ and note that
$pr_U\circ \tau=p_U$. Take the monic polinomial
$h \in \mathcal O[t]$ from the item (f) of Theorem \ref{ElementaryNisSquare_1}.
With this replacement of notation and with the element $h$ we arrive to the following

\begin{thm}[stronger than Theorem \ref{MajorIntrod}]\label{Major}
Let the field $k$, the variety $X$, its closed points $x_1,x_2,\dots,x_n$, the semi-local ring $\mathcal O=\mathcal O_{X,\{x_1,x_2,\dots,x_n\}}$
the semi-local scheme $U=Spec(\mathcal O)$,
the function $\textrm{f}\in k[X]$
be the same as in Theorem \ref{ElementaryNisSquare_1}.
%Given a non-zero function $\textrm{f}\in k[X]$ vanishing at each point $x_i$,
Then the monic polinomial $h\in O[t]$,
the commutative diagram
of schemes with the irreducible affine $U$-smooth $Y$
%of the form
%(\ref{SquareDiagram2_2})
\begin{equation}
\label{SquareDiagram2_2}
    \xymatrix{
       (\Aff^1 \times U)_{h}  \ar[d]_{inc} && Y_h:=Y_{\tau^*(h)} \ar[ll]_{\tau_{h}}  \ar[d]^{inc} \ar[rr]^{(p_X)|_{Y_h}} && X_f  \ar[d]_{inc}   &\\
     (\Aff^1 \times U)  && Y  \ar[ll]_{\tau} \ar[rr]^{p_X} && X                                     &\\
    }
\end{equation}
the morphism $\delta: U\to Y$ subject to the following conditions:
\begin{itemize}
\item[\rm{(i)}]
the left hand side square
is an elementary {\bf distinguished} square in the category of affine $U$-smooth schemes in the sense of
\cite[Defn.3.1.3]{MV};
\item[\rm{(ii)}]
$p_X\circ \delta=can: U \to X$, where $can$ is the canonical morphism;
\item[\rm{(iii)}]
$\tau\circ \delta=i_0: U\to \Aff^1 \times U$ is the zero section
of the projection $pr_U: \Aff^1 \times U \to U$;
\item[\rm{(iv)}] $h(1)\in \mathcal O[t]$ is a unit.
\end{itemize}
\end{thm}

\begin{proof}
The items \rm{(i)} and \rm{(iv)} of the Corollary \ref{ElementaryNisSquareNew_1}
show that the morphisms $\delta: U\to Y$ and $(p_X)|_{Y_h}: Y_h\to X_f$ are well defined.
%The item \rm{(iv)} of the Corollary \ref{ElementaryNisSquareNew_1}
%shows that the morphism  is well defined.
The items \rm{(i)}, \rm{(ii)} of that Corollary show that the left hand side square in the diagram
(\ref{SquareDiagram2_2})
is an elementary {\bf distinguished} square in the category of smooth $U$-schemes in the sense of
\cite[Defn.3.1.3]{MV}. The equalities
$p_X\circ \delta=can$ and $\tau\circ \delta=i_0$ are clear.
The property (iv) follows from the items (e), (f) and (a)
of Theorem \ref{ElementaryNisSquare_1}.
\end{proof}

\begin{rem}
\label{a_frame}
The latter commutative diagram gives rise to a pointed motivic space morphism
$$\alpha: \mathbb P^1\times U/{\infty} \times U \to \Aff^1\times U/(\Aff^1 \times U)_{h} \xrightarrow{(p_X)|_{Y_h}} \to X/X_f$$
such that $\alpha|_{0\times U}: 0\times U \to X/X_f$ coincides with the morphism
$U\xrightarrow{can} X \to X/X_f$.
\end{rem}

\section{A moving lemma}\label{An extended_moving_lemma_sect}
Let $k$ be a field. Particularly, $k$ can be a finite field. We prove the following useful geometric theorem
\begin{thm}[A moving lemma]
\label{A_moving_lemma}
Let $X$ be a $k$-smooth quasi-projective irreducible $k$-variety, and let $x_1,x_2,\dots,x_n$ be closed points in $X$.
Let $U=Spec(\mathcal O_{X,\{x_1,x_2,\dots,x_n\}})$.
Let $Z$ be a closed subset in $X$.
Let
$U\xrightarrow{can} X$ be the inclusion and
$X\xrightarrow{p} X/(X-Z)$
be the factor morphism.
Let $*=(X-Z)/(X-Z) \in X/(X-Z)$
be the distinguished point of $X/(X-Z)$.
Given a closed subset $Z\subset X$
%Given a non-zero function $\textrm{f}\in k[X]$, vanishing at each the point $x_i$
there is a Nisnevich sheaf morphism
$$\Phi_t: \Aff^1 \times U \to X/(X-Z)$$
such that $\Phi_0: U\to X/(X-Z)$ is the composite morphism
$U\xrightarrow{can} X\xrightarrow{p} X/(X-Z)$ and $\Phi_1: U\to X/(X-Z)$ takes $U$ to
the distinguished point $*$ in $X/(X-Z)$.
\end{thm}

\begin{proof}
Take a function $\text{f} \in k[X]$ such that $\text{f}$
vanishes as on $Z$, so at all the points
$x_1,x_2,\dots,x_n$.
Consider the commutative diagram (\ref{SquareDiagram2_2}).
Since the left hand side square is a distinguished elementary square,
hence the morphism
$\sigma: Y/Y_h \to \Aff^1_U/(\Aff^1 \times U)_{h}$
of Nisnevich sheaves is an isomorphism.
Thus there is a composite morphism of motivic spaces of the form
$$\Phi_t: \Aff^1_U\to \Aff^1_U/(\Aff^1 \times U)_{h} \xrightarrow{\sigma^{-1}}
Y/Y_h   \xrightarrow{p_X}  X/X_{\textrm{f}}.$$
Let $i_0: 0\times U \to \Aff^1_U$ be the natural morphism. By the properties (a') and (d) from
Theorem
\ref{ElementaryNisSquare_1}
the morphism
$\Phi_0:=\Phi \circ i_0$ equals to the one
$$U\xrightarrow{can} X \xrightarrow{p} X/X_{\textrm{f}},$$
where
$p: X\to X/X_{\textrm{f}}$
is the canonical morphisms
By the item (e) of Theorem \ref{ElementaryNisSquare_1}
the morphism
$\Phi_1:=\Phi  \circ i_1: U\to X/X_f$
is the constant morphism to
the distinguished point $*$ of $X/(X-Z)$.
\end{proof}

\begin{thm}[An extended moving lemma]
\label{An extended_moving_lemma}
Let $X$ be a $k$-smooth quasi-projective irreducible $k$-variety, and let $x_1,x_2,\dots,x_n$ be closed points in $X$.
Let $U=Spec(\mathcal O_{X,\{x_1,x_2,\dots,x_n\}})$. Let $c > 0$ be an integer.
Let $Z$ be a closed subset in $X$ of pure codimension $c$ in $X$.
Then there is a closed subset $Z^{new}$ in $X$ containing $Z$ and of pure codimension $c-1$
and a morphism of pointed Nisnevich sheaves
$$\Phi_t: \Aff^1 \times (U/(U-Z^{new}) \to X/(X-Z)$$
such that
$\Phi_0: U/(U-Z^{new})\to X/(X-Z)$ is the composite morphism \\
$U/(U-Z^{new})\xrightarrow{can} X/(X-Z^{new})\xrightarrow{p} X/(X-Z)$; and\\
$\Phi_1: U/(U-Z^{new})\to X/(X-Z)$ \\takes $U/(U-Z^{new})$ to
the distinguished point $*$ in $X/(X-Z)$.
\end{thm}

\begin{proof}
Take a function $\text{f} \in k[X]$ such that $\text{f}$
vanishes as on $Z$, so at all the points
$x_1,x_2,\dots,x_n$.
Consider the commutative diagram (\ref{SquareDiagram2_2}).
Recall \cite[Sect.6]{PSV} that additionally there is a commutative diagram of schemes
\begin{equation}
\label{SquareDiagram2_3}
    \xymatrix{
     Y \ar[rr]^{\rho} \ar[rrd]^{p_U} && \mathcal X   \ar[d]^{q_U} \ar[rr]^{q_X} && X  \ar[d]_{r_X}   &\\
    && U  \ar[rr]^{r_U} && S                                     &\\
    }
\end{equation}
with $\mathcal X=U\times_S X$ such that $r_X$ is an almost elementary fibration such that
$r_X|_{\{f=0\}}: \{f=0\}\to S$ is finite,
$q_X$ and $q_U$ are the projections, and $rho$ is \'{e}tale.

Set $r_U=r_X\circ can$,
$p_U=q_U\circ \rho$,
$p_X=q_X\circ \rho$.
Then $p_U$
is a smooth morphism of relative dimension 1 with an irreducible $Y$.
Set $Z^{new}:=r^{-1}_X(r_X(Z))$.
Since $r_X|_{\{f=0\}}: \{f=0\} \to S$ is finite and $Z$ is in $\{f=0\}$, hence
$Z$ is finite over $S$ and $r_X(Z)$ is equi-dimensional.
Since $r_X$ is an almost elementary fibration, hence
$Z^{new}$ is equi-dimensional of pure codimension $c-1$ in $X$.
Clearly, $Z$ is in $Z^{new}$. Set $\mathcal Z=p^{-1}_X(Z)$.
Then in the diagram the left hand side square
\begin{equation}
\label{SquareDiagram2_4}
    \xymatrix{
       (\Aff^1 \times U)-\tau(\mathcal Z)  \ar[d]_{inc} && Y-\mathcal Z \ar[ll]_{\tau|_{Y-\mathcal Z}}  \ar[d]^{inc} \ar[rr]^{(p_X)|_{Y_h}} && X-Z  \ar[d]_{inc}   &\\
     (\Aff^1 \times U)  && Y  \ar[ll]_{\tau} \ar[rr]^{p_X} && X                                     &\\
    }
\end{equation}
is an elementary distinguished square. Hence $\tau$ induces a Nisnevich sheaf isomorphism
$\tau: Y/(Y-\mathcal Z) \to (\Aff^1 \times U)/((\Aff^1 \times U)-\tau(\mathcal Z))$.
Note that for $Z^{new}_U:=can^{-1}(Z^{new})$ one has
$p^{-1}_U(Z^{new}_U)=p^{-1}_X(Z^{new})$
and $\tau^{-1}(\Aff^1\times Z^{new}_U)=p^{-1}_U(Z^{new}_U)$.
Let
$$\Phi_t: \Aff^1 \times (U/U-Z^{new}_U) \xrightarrow{\Pi} (\Aff^1 \times U)/((\Aff^1 \times U)-\tau(\mathcal Z))\xrightarrow{\tau^{-1}}  Y/(Y-\mathcal Z)\xrightarrow{p_X}  X/(X-Z)$$
be the composite morphism.

It is straightforward to check that
$\Phi_0: U/(U-Z^{new})\to X/(X-Z)$ is the composite morphism
$U/(U-Z^{new})\xrightarrow{can} X/(X-Z^{new})\xrightarrow{p} X/(X-Z)$.
The property (iv) of the polinomial $h$ from
Theorem \ref{Major} yields
that the morphism
$\Phi_1: U/(U-Z^{new})\to X/(X-Z)$ takes $U/(U-Z^{new})$ to
the distinguished point $*$ in $X/(X-Z)$.
\end{proof}

\section{Coisin complex and strict homotopy invariance}\label{cous_complex_section}
Let $k$ be a field.
Let $A: SmOp/k \to Gr-Ab$ is a cohomology theory on the category $SmOp/k$
in the sense of \cite[Sect. 1]{PS}.
Let $\mathcal O$ be the semi-local ring of finitely many {\bf closed points} on a
$k$-smooth irreducible affine $k$-variety $X$, $d=dim X$. Let $U=\text{Spec}(\mathcal O)$.
The main result of the present section and of the paper states that
the Cousin complex of $U$ associated with the theory $A$ is exact.
A relative version of that result is proved too.
\begin{thm}\label{cousin_1}
For any integer $n$ the Cousin complex
$$0\to A^n(U) \to A^n(\eta) \xrightarrow{\partial} \oplus_{x\in U^{(1)}} A^{n+1}_{x}(U) \xrightarrow{\partial} ... \xrightarrow{\partial} \oplus_{x\in U^{(d)}}A^{n+d}_{x}(U) \to 0$$
is exact. If $d=1$ and $\mathcal O$ is local, then the sequence
$0\to A^n(U) \to A^n(\eta) \xrightarrow{\partial} A^{n+1}_{x}(U) \to 0$
is exact.
\end{thm}

\begin{proof}
It follows in the standard manner from Theorem \ref{An extended_moving_lemma}
and the definition of a cohomology theory on $SmOp/k$.
\end{proof}

\begin{cor}\label{cousin_1_cor}
The Zariski sheaf $\mathcal A^n_{Zar}$ on $X$, associated with the presheaf $W\mapsto A^n(W)$ has a flask resolution of the form
$$0\to \mathcal A^n_{Zar} \to \eta_*(A^n(\eta)) \xrightarrow{\partial} \oplus_{x\in X^{(1)}}(i_x)_* (A^{n+1}_{x}(X)) \xrightarrow{\partial} ... \xrightarrow{\partial}
\oplus_{x\in X^{(d)}}(i_x)_*(A^{n+d}_{x}(X)) \to 0.$$
\end{cor}

\begin{lem}
For any field $K$ containing $k$
%one has $H^0(\Aff^1_K,\mathcal A^n)=A^n(k)$, $H^0(\Aff^1_K,\mathcal A^n)=0$.
the sequence
$$0\to A^n(K) \xrightarrow{i} A^n(K(t)) \xrightarrow{\partial} \oplus_{x\in \Aff^1_K} A^{n+1}_x(\Aff^1_K) \to 0$$
is exact.
\end{lem}

\begin{proof}[Proof of the lemma]
By the definition $\partial=\Sigma  \partial_x$, where the sum runs
over all the closed points of $\Aff^1_K$. Since
$A^n(K)=A^n(\Aff^1_K)$, hence
the composite map
$\partial \circ i$ iz zero.
Particularly, $\partial_x \circ i=0$.
Hence $Im(i)$ is in $Ker(\partial_x)$.
By Theorem \ref{cousin_1} applied to $\text{Spec}(\mathcal O_{\Aff^1,x})$
one has $Im(i)$ is in $A^n(\mathcal O_{\Aff^1,x})$. Taking the origin of coordinates as the point $x$
and consider the composition map $i^*_0 \circ i: A^n(K) \to A^n(K)$, where
$i^*_0: A^n(\text{Spec}(\mathcal O_{\Aff^1,0})) \to A^n(\text{Spec}(K))$
is the pull-back map. Since $i^*_0 \circ i=id^*=id_{A^n(K)}$ is the identity,
hence $i$ is injective (for any integer $n$).
The sequence
$$A^n(\Aff^1_K) \xrightarrow{i} A^n(K(t)) \xrightarrow{\partial} \oplus_{x\in \Aff^1_K} A^{n+1}_x(\Aff^1_K) \to  A^n(\Aff^1_K) \xrightarrow{i} A^{n+1}(K(t))$$
is a part of the long exact sequence. The maps $i$ are injective.
Thus the maps
$\partial$ are surjective for any integer $n$.
\end{proof}

We need in the following modification of Theorem \ref{cousin_1}.
\begin{thm}\label{cousin_3}
Under the notation and hypotheses of Theorem \ref{cousin_1}
let $C$ be a $k$-smooth absolutely irreducible curve.
Let $x \in X$ be a point, $c\in C\times_k k(x)$ be a closed point and $\eta_x \in C\times_k k(x)$ be the generic point.
Then the cousin complex
$$0\to A^n(\mathcal O_{C\times X,c}) \xrightarrow{j^*} A^n(\mathcal O_{C\times X,\eta_x}) \xrightarrow{\partial_c} A^{n+1}_{x}(C\times X) \to 0$$
is exact. Particularly, $A^n(\mathcal O_{C\times X,c})=Ker(\partial_c)$ and there is a well-defined map
$c^*: Ker(\partial_c)=A^n(\mathcal O_{C\times X,c}) \to A^n(c)$.
\end{thm}

\begin{proof}
The localization sequence
$$A^n(\mathcal O_{C\times X,c}) \xrightarrow{j^*} A^n(\mathcal O_{C\times X,\eta_x}) \xrightarrow{\partial} A^{n+1}_{x}(C\times X)  \to A^{n+1}(\mathcal O_{C\times X,c})
\xrightarrow{j^*} A^{n+1}(\mathcal O_{C\times X,\eta_x})$$
is exact. Both maps $j^*$ are injective by Theorem \ref{cousin_1}.
Thus the map $\partial$ is surjective.
\end{proof}

\begin{cor}\label{cousin_4}[of Theorem \ref{cousin_3}]
Under the notation and hypotheses of Theorem \ref{cousin_1}
let $x \in X$ be a point and $\eta_x \in \Aff^1_{k(x)}$ be the generic point.
Then the complex
$$0\to A^n_{x}(X) \xrightarrow{j^*\circ p^*} A^n(\mathcal O_{\Aff^1\times X,\eta_x}) \xrightarrow{\partial} \oplus_{c\in \Aff^1_{k(x)}} A^{n+1}_{x}(\Aff^1\times X) \to 0$$
is exact.
\end{cor}

\begin{proof}[Proof of the corollary]
By the definition $\partial=\Sigma  \partial_c$, where the sum runs
over all the closed points of $\Aff^1_{k(x)}$.
%Since
%$A^n_{x}(X)=A^n_{\Aff^1_{k(x)}}(\Aff^1 \times X)$, hence
The composite map
$\partial \circ i$ iz zero.
Particularly, $\partial_c \circ i=0$.
Hence $Im(j^*\circ p^*)$ is in $Ker(\partial_c)$.

By Theorem \ref{cousin_3}
there is a well-defined pull-back map
$c^*: Ker(\partial_c)=A^n(\mathcal O_{C\times X,c}) \to A^n(c)$.
Taking $c$ to be the origing of coordinates on $\Aff^1_{k(x)}$,
we see that $i^*_0 \circ j^* \circ p^*=id^*=id_{A^n_{x}(X)}$.
Hence the map $j^* \circ p^*$ is injective. The map
$p^*: A^n_{x}(X) \to A^n_{\Aff^1_{k(x)}}(\Aff^1 \times X)$
is an isomorphism by the homotopy invariance of the theory $A$. Thus the map $j^*$ is injective.
The sequence
$$A^n_{\Aff^1_{k(x)}}(\Aff^1 \times X) \xrightarrow{j^*} A^n(\mathcal O_{\Aff^1\times X,\eta_x}) \xrightarrow{\partial} \oplus_{c\in \Aff^1_{k(x)}} A^{n+1}_x(\Aff^1 \times X)
\xrightarrow{}  A^{n+1}_{\Aff^1_{k(x)}}(\Aff^1 \times X) \xrightarrow{j^*} A^{n+1}(\mathcal O_{\Aff^1\times X,\eta_x})$$
is a part of the long exact sequence. The maps $j^*$ are injective.
Thus the maps
$\partial$ are surjective for any integer $n$.
Whence the lemma.
\end{proof}

\section{Strict homotopy invariance of sheaves $\mathcal A^n_{Nis}$ }\label{StrHomInvSect}
For basic definitions used in this section see \cite[Sect.1]{PS}. We will
work below with $\mathbb Z$-graded cohomology theories on the category
$SmOp/k$ and will suppose that the boundary maps $\partial_{(X,U)}$
have degree $+1$.
\begin{thm}\label{StrHomInv}
Let $k$ be a field and let $A: SmOp/k \to Gr-Ab$ be a cohomology theory on
the category $SmOp/k$ in the sense of \cite[Sect.1]{PS}. Let
$\mathcal A^n_{Zar}$ be the Zariski sheaf associated with the presheaf
$W\mapsto A^n(W)$. Then $\mathcal A^n_{Zar}$ is homotopy invariant
and even it is strictly homotopy invariant on $(Sm/k)_{Zar}$.
\end{thm}

\begin{proof}
Let $X$ be a $k$-smooth variety. Let $p: \Aff^1 \times X \to X$ be
the projection. The morphism $p$ induces the pull-back
morphisms of the Cousin complexes
$p^*: Cous(X) \to Cous(\Aff^1\times X)$.
By Corollary \ref{cousin_4} that morphism is a quasi-isomorphism of complexes of abelian groups.
By Corollary \ref{cousin_1_cor} complexes
$Cous(X)$ and $Cous(\Aff^1\times X)$
computes the Zarisky cohomology of the sheaf $\mathcal A^n$
on $X$ and on $\Aff^1\times X$ respectively.
And the the pull-back
morphisms of the Cousin complexes
$p^*$ induces the pull-back map
$$p^*: H^r_{Zar}(X, \mathcal A^n_{Zar}) \to H^r_{Zar}(\Aff^1\times X, \mathcal A^n_{Zar}).$$
Hence the latter map is an isomorphism.
Whence the theorem.
\end{proof}

\begin{thm}\label{StrHomInvÒøû}
Let $k$ be a field and let $A: SmOp/k \to Gr-Ab$ be a cohomology theory on
the category $SmOp/k$ in the sense of \cite[Sect.1]{PS}. Let
$\mathcal A^n_{Nis}$ be the Nisnevich sheaf associated with the presheaf
$W\mapsto A^n(W)$. Then for any $X\in Sm/k$ one has $\mathcal A^n_{Zar}(X)=\mathcal A^n_{Nis}(X)$,
$H^p_{Zar}(X,\mathcal A^n_{Zar})=H^p_{Nis}(X,\mathcal A^n_{Nis})$.

Particularly, the Nisnevich sheaf $\mathcal A^n_{Nis}$
is strictly homotopy invariant on $(Sm/k)_{Nis}$.
\end{thm}

\begin{proof}
It is derived in a standard manner from
Corollary \ref{cousin_1_cor} and Theorem \ref{StrHomInv}.
\end{proof}

\section{Appendix A: proof of Lemma \ref{SmallAmountOfPoints}}\label{Appendix}
Let $k$ be a {\bf finite} field. Let $\mathcal O$ be the semi-local ring of finitely many {\bf closed points} on a
$k$-smooth irreducible affine $k$-variety $X$.
Set $U=\spec \mathcal O$. Let ${\bf u}\subset U$ be the set of all closed points in $U$.
For a point $u\in {\bf u}$ let $k(u)$ be its residue field.

\begin{notation}\label{notn: F1F2}
Let $k$ be the finite field of characteristic $p$,
$k'/k$ be a finite field extensions.
Let $c=\sharp(k)$ (the cardinality of $k$).
For a positive integer $r$ let $k'(r)$
%(respectively $\mathbb L(r)$, respectively $\mathbb E(r)$)
be a unique field extension of the degree $r$ of the field $k'$.
%(respectively of $\mathbb L$, respectively of $\mathbb E$).
Let $\mathbb A^1_{k}(r)$ be the set of all degree $r$ points
on the affine line $\mathbb A^1_{k}$.
Let $Irr(r)$ be the number of the degree $r$ points on
$\mathbb A^1_{k}$.

\end{notation}

\begin{lem}\label{l:F1F2_preliminary}
Let $k$ be the finite field of characteristic $p$,
$c=\sharp(k)$,
$k^{\prime}/k$ be a finite field extension of degree $d$.
%$m$ a positive integer,
%$A=\prod^m_{i=1}\mathbb L$ the commutative $\mathbb L$-algebra.
%For a positive integer $r$ let $\mathbb F(r)$ (respectively $\mathbb L(r)$)
%be a unique field extension of the degree $r$ of the field $\mathbb F$ (respectively of $\mathbb L$).
Let $q\in \mathbb N$
be {\bf a prime} which is {\bf co-prime} as to the characteristic $p$ of the field $k$,
so to the integer $d$.
%Then for any integer $n\gg 0$ the following hold\\
%For a positive integer $r$ let $Irr_{\mathbb F}(r)$ be the number of degree $r$ points on the affine line $\mathbb A^1_{\mathbb F}$.
Then \\
(1)$Irr(q)=(c^q-c)/q$;\\
(2) the $k$-algebra $k^{\prime}\otimes_{k} k(q)$ is the field $k^{\prime}(q)$;\\
(3) $Irr(dq) \geq Irr(q)$;\\
(4) $Irr(dq)\geq  (c^q-c)/q \gg 0$ for $q\gg 0$.
%(2) there is a surjective $\mathbb F$-algebra homomorphism
%$\alpha_{q^n}: \mathbb F[t] \to A\otimes_{\mathbb F} \mathbb F(q^n)=\prod^m_{i=1}\mathbb L(q^n)$.
%%$V\times_u u(q^n) \hookrightarrow \mathbb A^1_u$
\end{lem}

\begin{proof}
The assertions (1) and (2) are clear. The assertion (3) it is equivalent
to the unequality $\varphi(c^{dq}-1)/dq \geq \varphi(c^q-1)/q$. The latter is equivalent
to the one $\varphi(c^{dq}-1)/d \geq \varphi(c^q-1)$.

The norm map $N: k(dq)^{\times} \to k(q)^{\times}$ is surjective.
Both groups are the cyclic groups of orders $c^{dq}-1$ and $c^q-1$ respectively.
Thus for any $c^q-1$-th primitive root of unity $\xi \in k(q)^{\times}$
there is a $c^{dq}-1$-th primitive root of unity $\zeta \in k(dq)^{\times}$
such that $Norm(\zeta)=\xi$. On the other hand, if $N(\zeta)=\xi$
and $\sigma \in Gal(k(dq)/k(q))$, then $N(\zeta^{\sigma})=\xi$.
Thus $\varphi(c^{dq}-1)/d \geq \varphi(c^q-1)$. The assertion (3) is proved.

The assertions (1) and (3) yield the assertion (4).
\end{proof}

\begin{notation}
\label{d(Y_u)}
For any \'{e}tale $k$-scheme $W$
set $d(W)=\text{max} \{ deq_{k} k(v)| v\in W \}$.
\end{notation}

\begin{lem}\label{l:F1F2_surjectivity}
Let $Y_u$ be an \'{e}tale $k$-scheme. For any positive integer $d$ let
$Y_u(d)\subseteq Y_u$ be the subset consisting of points $v\in Y_u$ such that
$deg_{k}(k(v))=d$.
For any prime $q\gg 0$ the following holds: \\
(1) if $v\in Y_u$, then $k(v)\otimes_{k} k(q)$ is the field $k(v)(q)$ of the degree $q$ over $k(v)$; \\
(2) $k[Y_u]\otimes_{k} k(q)=(\prod_{v\in Y_u}k(v))\otimes_{k} k(q) =\prod_{v\in Y_u} k(v)(q)$;\\
(3) there is a surjective
$k$-algebra homomorphism
\begin{equation}\label{eq:first_surjectity}
\alpha: k[t]\to k[Y_u]\otimes_{k} k(q)
\end{equation}
\end{lem}

\begin{proof}
The first assertion follows from the lemma \ref{l:F1F2_preliminary}(2).
The second assertion follows from the first one.
Prove now the third assertion.
%Set $d(Y_u)=\text{max}\{ deq_{\mathbb F} \mathbb F(v)| v\in Y_u \}$.
The $k$-algebra homorphism
$$k[t]\to \prod^{d(Y_u)}_{d=1}(\prod_{x\in \mathbb A^1_{k}(dq))}k[t]/\mathfrak(m_x))=
\prod^{d(Y_u)}_{d=1}(\prod_{x\in \mathbb A^1_{k}(qd)} k(dq))$$
is surjective. The $k$-algebra $k[Y_u]$ is equal to
$\prod^{d(Y_u)}_{d=1}\prod_{v\in Y_u(d)}k(v)=\prod^{d(Y_u)}_{d=1}\prod_{v\in Y_u(d)}k(d)$.
Thus for any prime $q\gg 0$ one has
$$k[Y_u]\otimes_{k} k(q)=\prod^{d(Y_u)}_{d=1}\prod_{v\in Y_u(d)}k(dq).$$
Choose a prime $q\gg 0$ such that for any $d=1,2,...,d(Y_u)$ one has
$Irr(dq)\geq \sharp(Y_u(d))$.
This is possible by the lemma \ref{l:F1F2_preliminary}.
In this case for any $d=1,2,...,d(Y_u)$ there exists
a surjective $k$-algebra homomorphism
$\prod_{x\in \mathbb A^1_{k}(qd)} k(dq)\to \prod^{d(Y_u)}_{d=1}\prod_{v\in Y_u(d)}k(dq)$.
Thus for this specific choice of prime $q$ there exists a surjective $k$-algebra homomorphism
$k[t]\to k[Y_u]\otimes_{k} k(q)$.
The third assertion of the lemma is proved.
%The second assertion follows from the first one since the ideals $(t)$, $(g))$ and $Ker(\alpha)$ are pair-wise co-prime.
\end{proof}

%%%\begin{lem}
%%%\label{F1F2_Appendix}
%%%Let $U$ be as in the Proposition.
%%%Let $Z\subset\mathbb A^1_U$ be a closed subscheme finite over $U$.
%%%Let $Y^{\prime} \to U$ be a finite \'{e}tale morphism such that
%%%for any closed point $u_i$ in $U$ the fibre $Y^{\prime}_{u_i}$ of $Y^{\prime}$ over $u_i$
%%%contains a $k(u_i)$-rational point. Then there are finite field extensions
%%%$k_1$ and $k_2$ of the finite field $k$ such that \\
%%%(i) the degrees $[k_1: k]$ and $[k_2: k]$ are coprime,\\
%%%(ii) $k(u_i) \otimes_k k_r$ is a field for $r=1$ and $r=2$,\\
%%%(iii) the degrees $[k_1: k]$ and $[k_2: k]$ are strictly greater than any of the degrees
%%%$[k(z): k(u)]$, where $z$ runs over all closed points of $Z$,\\
%%%(iv) there is a closed embedding of $U$-schemes
%%%$Y^{\prime\prime}=((Y^{\prime}\otimes_k k_1) \coprod (Y^{\prime}\otimes_k k_2)) \xrightarrow{i} \mathbb A^1_U$,\\
%%%(v) for $Y=i(Y^{\prime\prime})$ one has $Y \cap Z = \emptyset$,\\
%%%(vi) for any closed point $u_i$ in $U$ one has
%%%$Pic(\P^1_{u_i}-Y_{u_i})=0$.
%%%\end{lem}

Under the notation \ref{notn: F1F2} and \ref{d(Y_u)} state one more lemma,
which is used to prove the lemma \ref{SmallAmountOfPoints}.
\begin{lem}\label{l:S'_and_S''_preliminary}
Let $l/k$ be a finite field extention. Let $V$ be an \'{e}tale $l$-scheme
containing a $l$-rational point $v_0$. Let $V':=V-\{v_0\}$. Then for any prime $q\gg 0$ the following holds \\
(1) There is a surjective $k$-algebra map
$$\beta: k[t]\to [l[V']\otimes_{k} k(q)]\times [l[v_0] \times l(q-1)];$$
(2) the \'{e}tale extensions
$l[V']\hookrightarrow l[V']\otimes_{k} k(q)$
and
$in: l \hookrightarrow [l[v_0] \times l(q-1)]$
are of the form
$l[V'][t]/(f(t))$ and $l[t]/(t\cdot g)$ respectively. Here
$f(t)\in l[V'][t]$
is a monic degree $q$ polinomial
and
$g(t)\in l[t]$
a monic degree $(q-1)$ irreducible polinomial;\\
(3) the projection $p: [l[v_0] \times l(q-1)]\to l[v_0]=l$ is such that $p\circ in=id_E$.
%%%If additionally $q > d(Y_u)+1$, then for any degree
%%%$q-1$ irreducible monic polinomial $g(t)\in \mathbb F[t]$
%%%the $\mathbb F$-algebra homomorphism
%%%\begin{equation}\label{eq:second_surjectivity}
%%%\beta: \mathbb F[t]\to [\mathbb F[t]/(t) \times  \mathbb F[t]/(g)] \times [\mathbb F[Y_u]\otimes_{\mathbb F} \mathbb F(q)],
%%%\end{equation}
%%%taking the variable $t$ to $(t~\text{mod} (t), t~\text{mod} (g),\alpha(t))$ is surjective.
\end{lem}

\begin{proof}
For any positive integer $d$ let
$V'(d)\subseteq V'$ be the subset consisting of points $v\in V'$ such that
$deg_{l}(l(v))=d$.
Let $d(V')=\text{max}\{ deq_{l} (l(v))| v\in V' \}$.
Let $d_0=deg_k(l)$.

Choose a prime $q\gg 0$ such that for any $d=1,2,...,d(V')$ one has
$Irr_{k}(dq)\geq \sharp(V'(d))$.
This is possible by the lemma \ref{l:F1F2_preliminary}(4).
Regarding the $l$-scheme $V'$ as the $k$-scheme and
applying the lemma
\ref{l:F1F2_surjectivity}(3)
we get a surjective $k$-algebra map
$\beta': k[t]\to l[V']\otimes_{k} k(q)=\prod_{v\in V'} k(v)(q)$.
Now consider a $k$-algebra map
$$\beta=(\beta',\beta_2,\beta_3): k[t]\to [l[V']\otimes_{k} k(q)]\times [l[v_0] \times l(q-1)],$$
where $\beta_2,\beta_3$ are surjective.
Since $q\gg 0$ and $q$ is co-prime to $(q-1)$, hence for any integer
$d=1,2,...,d(V')$ one has
$d\cdot q\neq d_0$, $dq\neq d_0\cdot(q-1)$ and $d_0\neq d_0\cdot (q-1)$.
The chine's remainder theorem
yields now the surjectivity of $\beta$.

The second and the third assertions are obvious.
\end{proof}

\begin{proof}[Proof of the lemma \ref{SmallAmountOfPoints}]
We prove this lemma for the case of local $U$ and left the general case to the reader.
So there is only one closed point $u\in U$. Let $k(u)$ be its residue field. It is a finite field
extension of the finite field $k$.
Let $Y$ be the set of closed points of the scheme $S'$ and $V$ be the scheme $\sqcup_{v\in Y} v$.
Set $v_0=\delta(u)$ and $V'=V-\{v_0\}$.
Then $V$, $V'$ and $v_0$ are \'{e}tale $k(u)$-schemes
and $v_0$ is a $k(u)$-rational point in $V$.

Set $l=k(u)=k(v_0)$.
By the lemma \ref{l:S'_and_S''_preliminary} there is the prime number $q\gg 0$ and
the surjective $k$-algebra map
$\beta: k[t]\to [l[V']\otimes_{k} k(q)]\times [l[v_0] \times l(q-1)]$.
{\it We will use $\beta$ at the end of this proof}.

By the lemma \ref{l:S'_and_S''_preliminary}(2) one have equalities of the \'{e}tale $l$-algebras
$$l[V']\otimes_{k} k(q)=l[V'][t]/(f) \ \text{and} \ l[v_0] \times l(q-1)=l[t]/(t\cdot g).$$
Here
$f(t)\in l[V'][t]$
is a monic degree $q$ polinomial
and
$g(t)\in l[t]$
a monic degree $(q-1)$ irreducible polinomial.

The closed embedding $\delta|_u: u\hookrightarrow T'$ induces a surjection
$(\delta|_u)^*: k[T']\to k(u)=l$ of the $k$-algebras.
Choose a monic degree $(q-1)$ polinomial $\tilde g(t)\in k[T']$,
which is a lift of $g(t)\in l[t]$. Consider the
$l[V']\times k[T']$-algebra
$(l[V']\times k[T'])[t]/(f,t\cdot \tilde g)(l[V']\times k[T'])[t]$.

The closed embedding
$i_V\sqcup i_T: V'\sqcup T' \hookrightarrow S'$
induces a surjection
$(i_V\sqcup i_T)^*: k[S']\to l[V']\times k[T']$.
Choose a monic degree $q$ polinomial $F(t)\in k[S']$,
which is a lift of
$(f,t\cdot \tilde g)\in l[V'][t]\times k[T'][t]=(l[V']\times k[T'])[t]$.
Consider the $k[S']$-algebra $k[S'][t]/(F(t))$.
Set $$S''=Spec( \ k[S'][t]/(F(t)) \ ).$$
The inclusion $k[S']\to k[S'']=k[S'][t]/(F(t))$ is a $k[U]$-algebra homomorphism.
Let $\rho: S''\to S'$ be the corresponding morphism of the $U$-schemes.
Nakayama lemma yields that the closed subscheme $\rho^{-1}(T')\subseteq S''$ coincides with
$$Spec(k[T'][t]/(t\cdot \tilde g)=Spec(k[T'][t]/(t)\sqcup Spec(k[T'][t]/(\tilde g))=T'\sqcup Spec(k[T'][t]/(\tilde g))$$
and the morphism $\rho|_{T'}: T' \to S'$ coincides with the closed embedding $i_T: T'\hookrightarrow S'$.
Denoting $\delta': T' \hookrightarrow S''$ the inclusion we see that
$\delta'$ is a section of $\rho$ over $T' \subset S''$.

Clearly, the the morphism $\rho$ is finite flat.
The $l$-scheme $\rho^{-1}(V)$ coincides with the one
$Spec(l[V']\otimes_{k} k(q))\sqcup [Spec(l[v_0]\times l(q-1))]$.
It is \'{e}tale of degree $q$ over the $l$-scheme
$$V=Spec(l[V'])\sqcup Spec(l).$$
Hence the morphism $\rho$ {\it is finite \'{e}tale}.
The set of points
of the scheme $\rho^{-1}(V)$
coincides with
the set $\{z \in S^{\prime\prime}_u| [k(z):k(u)] < \infty \}$.
The surjectivity of the homomorphism $\beta$ shows now that
the latter set satisfies to the condition (2) of the lemma
\ref{SmallAmountOfPoints}.

Clearly, $v_0$ is the only $l$-rational point of the scheme $\rho^{-1}(V)$.
Thus the morphisms $\rho$ and $\delta'$ satisfy to the condition (1) of the lemma.
Whence the lemma \ref{SmallAmountOfPoints}.
\end{proof}

\section{Appendix B: proof of the proposition \ref{ArtinsNeighbor}}\label{Appendix_B}
Let $\mathbb F_q$ be a finite field of $q = p^a$ elements.
Let $S = \mathbb F_q[x_0,...,x_n]$ be the homogeneous coordinate ring of $\mathbb P^n$,
let $S_d \subset S$ be the $\mathbb F_q$-subspace of homogeneous polynomials of degree $d$,
and let $S_{homog} = \cup^{\infty}_{d=0} S_d$.
For each $f \in S_d$, let $H(f)$ be the subscheme $Proj(S/(f)) \subset \mathbb P^n$.
Typically (but not always), $H(f)$ is a hypersurface of dimension $n-1$ defined by the equation $f=0$.
The following Bertini type proposition is an extension of Artin's
result~\cite[Exp. XI, Thm. 2.1]{LNM305}.
It is a straightforward corollary of two theorems \cite[Thm. 1.2]{Poo1}, \cite[Thm. 1.1]{Poo2}.
\begin{prop}\label{Bertini_for_geom_connected}
Let $X$ be a projective smooth geometrically irreducible subscheme of $\mathbb P^n$ over $\mathbb F_q$. Let $Z$ be a finite subset of $X$,
Let $m\geq 2$ be the dimension of $X$.
Then there in an integer $N_0>0$ such that for any integer $d\geq N_0$ there is an $f\in S_d$ such that the scheme
$H(f)\cap X$ is smooth geometrically irreducible of dimension $m-1$.
\end{prop}

\begin{proof}
We give a proof in the simplest case, when the set $Z$ consists of one point and this point is rational.
The general case we left to the reader. So, we may and will assume that
$Z=x_0=[1:0:...:0] \in X$.
Let $\tau\subset \mathbb P^n$ be the tangent space to $X$ at the point
$x_0$.
In this case let $l=l(t_0,t_1,\dots, t_n)$ be a linear form in $\mathbb F_q[t_0,t_1,\cdots,t_n]$
such that
$l|_{\tau}\neq 0$.
Set
$$\mathcal P_1 :=\{f \in S_{homog} : H(f)\cap (X-x_0) \ \text{ \ is \ smooth \ of \ dimension} \ \ m-1, \ \text{and} f(x)=0, f_1=l \}.$$
By \cite[Thm. 1.2]{Poo1} one has $\mu(\mathcal P)=\frac{1}{q}\zeta_{(X-x_0)}(m+1)^{-1}>0$.
On the other hand the density of
$\mathcal P_2:=\{f\in S_{homog} : H(f)\cap X \text{is \ geometrically \ irreducible} \}$
is $1$. These yield that
there in an integer $N_0>0$ such that for any integer $d\geq N_0$ there is an $f\in S_d$ with $f(x)=0$, $f_1=l$ such that the scheme
$H(f)\cap (X-x_0)$ is smooth geometrically irreducible of dimension $m-1$. Since $f(x)=0$ and $l|_{\tau}\neq 0$, hence
the scheme $H(f)\cap X$ contains the point $x_0$ and it is smooth at this point too. Whence the proposition.
\end{proof}

We next prove the proposition \ref{ArtinsNeighbor}.
Recall thar it extends a result of Artin from
\cite{LNM305} concerning existence of nice neighborhoods.
%\begin{prop}
%\label{ArtinsNeighbor_copy} Let $k$ be {\bf a finite field}, $X$ be a smooth
%{\bf geometrically} irreducible {\bf affine} variety over $k$,
%$x_1,x_2,\dots,x_n\in X$ be a family of {\bf closed} points. Then there exists a
%Zariski open neighborhood $X^0$ of the family
%$\{x_1,x_2,\dots,x_n\}$ and {\bf an elementary fibration} $p:X^0\to
%S$, where $S$ is an open sub-scheme of the projective space
%$\Pro^{\mydim X-1}$.
%\par
%If, moreover, $Z$ is a closed co-dimension one subvariety in $X$,
%then one can choose $X^0$ and $p$ in such a way that $p|_{Z\bigcap
%X^0}:Z\bigcap X^0\to S$ is finite surjective.
%\end{prop}
The proof follows the proof of the original Artin's result. {\it However
we do not have in hand at the moment
any kind of theorem about existence of $d$ appropriative family of sections of the bundle $\mathcal O(d)$ for some $d\gg 0)$.
}
This is why there is a slight technical diffence beteen our proof and the original Artin's proof of his result.
The details are given below in the proof. Since the differences are purely technical we will give the proof for
the case of a smooth geometrically irreducible surface and left to the reader to recover the general case.
%we will use the proposition  \ref{Bertini_for_geom_connected} to find appropriative sections of bundles of the form
%$$\mathcal O(r_1), \ (r_1\gg 0), \ \mathcal O(r_1r_2), \ (r_1r_2r_3\gg 0),...,\mathcal O(r_1r_2...r_d), \ (r_1r_2...r_d\gg 0)$$
%rather then try to find $d$ appropriative sections of the bundle $\mathcal O(r)$ for some $r\gg 0)$.

\begin{proof}[Proof of the proposition \ref{ArtinsNeighbor} ]
We may and will assume that $X \subset \Aff^t_k$ is closed $k$-smooth geometrically irreducible surface.
%and still contains the points
%$x_1,x_2,\dots,x_n$.
Set $x:= \coprod^n_{j=1}x_i$.
Let $X_0$ be the closure of $X$ in $\Pro^t_k$.
Let $\bar X$ is the normalization of $X_0$ and set $Y=\bar X - X$ with the induced reduced structure.
Let $S \subset \bar X$ be the closed subset of $\bar X$ consisting of all singular points. Then  one has
\begin{itemize}
\item[(i)]
$S \subset Y$,
\item[(ii)]
$\dim \bar X = \dim X = 2$,
\item[(iii)]
$\dim Y = 1$,
\item[(iv)]
$\dim S \leq 0$ (so, $S$ consists of finitely many closed points in $Y$).
\end{itemize}
Embed $\bar X$ in a projective space $\Pro^r$.
%Let $M$ be the restriction of
%the invertible sheaf $\mathcal O_{\Pro^r}(1)$ to $\bar X$. Take the sheaf
%$M^{\otimes d_1}$ with the integer $d_1>N_0$, where the integer $N_0$ is from
%the proposition \ref{Bertini_for_geom_connected}.
By the proposition \ref{Bertini_for_geom_connected}
there is an integer $N_0>0$ such that for any integer $d_1\geq N_0$
there is there is an $f_1\in S_{d_1}$ such that \\
(0) for any ${j=1,...,n}$ one has $f_1(x_i)\neq 0$, \\
(1) the scheme $X_1:=H(f_1)\cap \bar X$ is geometrically irreducible of dimension $1$ and it is smooth at all the points, where $X$ is smooth,\\
(2) the scheme $Y_1:=H(f_1)\cap Y$ is of dimension $0$ and is smooth at all the points, where $Y$ is smooth;\\
(3) the scheme $S_1:=H(f_1)\cap S$ is empty (particularly, the scheme $Y_1$ is smooth).\\
The properties (1) and (3) show that $X_1$ is smooth.

By the  proposition \ref{Bertini_for_geom_connected}
there is an integer $N_1>0$ such that for any integer $d_1d_2\geq N_1$
there is there is an $f_2\in S_{d_1d_2}$ such that \\
($0_1$) for any ${j=1,...,n}$ one has $f_2(x_i)=0$, \\
($1_1$) the scheme $X_2:=H(f_2)\cap \bar X$ is geometrically irreducible of dimension $1$ and it is smooth at all the points, where $X$ is smooth,\\
($2_1$) the scheme $Y_2:=H(f_2)\cap Y$ is of dimension $0$ and is smooth at all the points, where $Y$ is smooth;\\
($3_1$) the scheme $S_2:=H(f_2)\cap S$ is empty (particularly, the scheme $Y_1$ is smooth),\\
($4_1$) the scheme $X_2\cap X_1$ is smooth of dimension $0$,\\
($5_1$) the scheme $X_2\cap X_1\cap Y$ is empty.\\
The properties ($1_1$) and ($3_1$) show that $X_2$ is smooth. The property ($5_1$)
shows that $X_2\cap X_1 \subset X$. Let $t_1,t_2$ be the homogeneous coordinates on the projective line
$\mathbb P^1_{\mathbb F_q}$.
Consider the section
$t_2\cdot f^{d_2}_1-t_1f_2$ on $\bar X\times \mathbb P^1_{\mathbb F_q}$
of the line bundle $\mathcal O(d_1d_2)|_{\bar X}\boxtimes \mathcal O_{\mathbb P^1(1)}$.
Let $\bar X'\subset \bar X\times \mathbb P^1_{\mathbb F_q}$
be a Cartier divisor defined by the equation $t_2\cdot f^{d_2}_1-t_1f_2=0$.
Let
$$\sigma=pr_{\bar X'}|_{\bar X'}: \bar X' \to \bar X.$$
It is easy to check that the scheme $\bar X'$ is a reduced and even normal.
One has
$\sigma^{-1}(X_1\cap X_2)=\mathbb P^1_{X_1\cap X_2}$.

Moreover, for any smooth open
$U$ in $\bar X$ the open subscheme $\sigma^{-1}(U)$ in $\bar X'$ is a smooth and open.
Particularly, $\sigma^{-1}(\bar X-S)$ and $\sigma^{-1}(X)=\sigma^{-1}(\bar X-Y)$ in $\bar X'$ are open and smooth.

\begin{lem}
Let $\bar X_{f_1}:=\bar X-X_1$ and  $\bar X'_{f_1}:=\sigma^{-1}(\bar X-X_1)$ and $\sigma_1:=\sigma|_{\bar X'_{f_1}}$.
Then $\sigma_1: \bar X'_{f_1}\to \bar X_{f_1}$ is an affine scheme isomorphism.

Futhermore, if
$Y_{f_1}=Y-Y_1$, then $Y_{f_1}$ is closed in $\bar X_{f_1}$ and we will identify
$Y_{f_1}$ with the closed subscheme $\sigma^{-1}_1(Y_{f_1})$ in $\bar X'_{f_1}$.
\end{lem}
\begin{proof}[Proof of the lemma]
In fact, both schemes are affine and the morphism $\sigma_1$ induces a $\mathbb F_q[\bar X_{f_1}]$-algebra
map
$$\sigma^*_1: \mathbb F_q[\bar X_{f_1}]\to \mathbb F_q[\bar X'_{f_1}]= \mathbb F_q[\bar X_{f_1}]/(t-f_2/f^{d_2}_1),$$
which is an isomorphism.
\end{proof}

We left to the reader the following lemma
\begin{lem}
The morphism $p: \bar X' \to \mathbb P^1_{\mathbb F_q}$ is flat. Let $X'_1\subset \bar X'$ be the closure of
$\sigma^{-1}(X_1-(X_1\cap X_2))$ in $\bar X'$ equipped with the reduced scheme structure. Then one has \\
(i) $\sigma^{-1}(X_1\cap X_2)=\mathbb P^1_{X_1\cap X_2}$;\\
(ii) $\sigma^{-1}(X_1)=\mathbb P^1_{X_1\cap X_2}\cup X'_1$;\\
(iii) $p^{-1}([0:1])=\mathbb P^1_{X_1\cap X_2}\cup X'_1=\sigma^{-1}(X_1)$;\\
(iv) $X'_1\cap \mathbb P^1_{X_1\cap X_2}=\{ \infty \}\times (X_1\cap X_2)$, where $\infty:=[0:1]\in \mathbb P^1_{\mathbb F_q}$.
\end{lem}
These two lemmas yield:
the morphism $\bar P=p|_{\bar X'-\bar X'_1}: \bar X'-\bar X'_1 \to \mathbb A^1_{\mathbb F_q}$ is projective and flat,
$\mathbb A^1_{X_1\cap X_2}$ is a closed subvariety in
$ (\bar X'-\bar X'_1)$,
$\sigma_1: (\bar X'-\bar X'_1)-\mathbb A^1_{X_1\cap X_2}=\bar X'_{f_1}\to \bar X_{f_1}$ is an affine scheme isomorphism.
Set $J=(\bar X_{f_1} \xrightarrow{\sigma^{-1}_1} (\bar X'-\bar X'_1)-\mathbb A^1_{X_1\cap X_2} \xrightarrow{In} \bar X'-\bar X'_1)$.
For $x \in \bar X_{f_1}$ set $P(x)=f_2(x)/f^{d_2}_1(x)$.
Summarizing we get a commutative diagram of the form
\begin{equation}
\label{SquareDiagram_2}
    \xymatrix{
     \bar X_{f_1} \ar[drr]_{P}\ar[rr]^{J}&&
\bar X'-\bar X'_1 \ar[d]_{\overline P}&& \mathbb A^1_{X_1\cap X_2} \ar[ll]_{I}\ar[lld]_{Q} &\\
     && \mathbb A^1_{\mathbb F_q}  &\\    }
\end{equation}
of morphisms satisfying the following conditions:
\begin{itemize}
\item[{\rm(i)}]
$J$ is the open immersion dense at each fibre of
$\overline P$, and $J(\bar X_{f_1})=(\bar X'-\bar X'_1)-\mathbb A^1_{X_1\cap X_2}$;
\item[{\rm(ii)}]
$\overline P$ is flat projective all of whose fibres
%are geometrically
%irreducible
equi-dimensional of dimension one;
\item[{\rm(iii)}]
$Q=id_{\mathbb A^1}\times r$, where $r: X_1\cap X_2\to Spec(\mathbb F_q)$
is the structure map.
\end{itemize}
Recall that by the property
($5_1$) the scheme $X_2\cap X_1\cap Y$ is empty.
Thus $Y_{f_1}=J(Y_{f_1})$ is closed even in $\bar X'-\bar X'_1$.
So, $Y_{f_1}$ as the $\mathbb A^1_{\mathbb F_q}$ is projective and affine.
Hence it is finite over $\mathbb A^1_{\mathbb F_q}$.

The scheme $\bar P^{-1}(\{0\})$ is isomorphic via the morphism $\sigma$ to $X_2$ and thus it is geometrically irreducible of dimension $1$ and smooth.
Hence the morphism $\bar P$ is smooth over a neighborhood $S$ of the point $0 \in \mathbb A^1_{\mathbb F_q}$.
By the property ($2_1$) after shrinking the neighborhood $S$ of the point $0 \in \mathbb A^1_{\mathbb F_q}$
we may and will assume that
$Y_{f_1}$ is finite and \'{e}tale over $S$.

Pulling back the diagram \ref{SquareDiagram_2} via the open embedding
$S\hookrightarrow \mathbb A^1_{\mathbb F_q}$
we get a commutative diagram of the form
\begin{equation}
\label{SquareDiagram_3}
    \xymatrix{
     X_{prel} \ar[drr]_{p_{prel}}\ar[rr]^{j_{prel}}&&
\bar X_{prel} \ar[d]_{\bar p_{prel}}&& S_{X_1\cap X_2} \ar[ll]_{i_{prel}}\ar[lld]^{q_{prel}} &\\
     && S  &\\    }
\end{equation}
of morphisms satisfying the following conditions:
\begin{itemize}
\item[{\rm(i)}]
$j_{prel}$ is the open immersion dense at each fibre of
$\bar p_{prel}$, and $X_{prel}=\bar X_{prel}-S_{X_1\cap X_2}$;
\item[{\rm(ii)}]
$\bar p_{prel}$ is smooth projective all of whose fibres
are geometrically
irreducible of dimension one;
\item[{\rm(iii)}]
$q_{prel}=id_{S}\times r$, where $r: X_1\cap X_2\to Spec(\mathbb F_q)$
is the structure map.
\item[{\rm(iv)}]
if $Y_S=\bar p^{-1}_{prel}(S)\cap Y_{f_1}$, then $Y_S$ is finite \'{e}tale over $S$ and $Y_S$ is in $X_{prel}$.
\end{itemize}
Taking now
$X_{fin}=X_{prel}-Y_S$, $Y_{fin}=S_{X_1\cap X_2}\sqcup Y_S$, $\bar X_{fin}=\bar X_{prel}$,
$p_{fin}=p_{prel}|_{X_{fin}}$, $\bar p_{fin}=\bar p_{prel}$, $q_{fin}=q_{prel}\sqcup (p_{prel}|_{Y_S})$,
$j_{fin}=j_{prel}|_{X_{fin}}$, $i_{fin}=i_{prel}|_{S_{X_1\cap X_2}}\sqcup j_{prel}|_{Y_S}$
we get a diagram of the form (\ref{SquareDiagram}) subjection to the conditions
\ref{DefnElemFib}(i) to \ref{DefnElemFib}(iii).
This proves that the morphism $p_{fin}: X_{fin} \to S$ is an elementary fibration.

To prove the last assertion of the proposition \ref{ArtinsNeighbor} it sufficient to
choose $f_1$ and $f_2$ such that they additionally satisfy the condition
$H(f_1)\cap H(f_2)\cap Z=\emptyset$.
Clearly, this is possible.

The proposition follows.
\end{proof}

\end{document}